\documentclass[10pt,leqno,oneside]{amsart}
\usepackage{amsmath,amssymb,amsthm}
\usepackage{bm,mathrsfs}
\usepackage{mathtools}
\usepackage{xspace}
\usepackage{color}
\usepackage{enumerate}
\usepackage{graphicx,subfigure,epstopdf}
\usepackage{hhline}
\usepackage[margin=1.05in]{geometry}
\usepackage{url}

\allowdisplaybreaks

\theoremstyle{plain}
\newtheorem{theorem}{Theorem}[section]
\newtheorem{remark}{Remark}[section]

\newtheorem{lemma}{Lemma}[section]

\newtheorem{example}{Example}[section]

\numberwithin{equation}{section}

\def\d{{\rm d}}
\def\al{\alpha}
\def\vth{\vartheta}
\def\rh{\varrho}
\def\luh{{\bar u}_h}
\def\K{\tau}
\def\zK{z^\K}
\def\Dal{{\partial^\alpha_t}}
\def\bDelh{{\bar{\Delta}_h}}

\def\dH#1{\dot H^{#1}(\Omega)}
\def\bPtau{\bar\partial_\tau}

\def\tu{\widetilde{u}}

\newcount\icount

\def\DDR#1#2{\icount=#1
  \ifnum\icount<1
 _{0}^{ \kern -.1em R} \kern -.2em \partial^{#2}_{\kern -.1em t}
  \else
 _{t}^{ \kern -.1em R} \kern -.2em \partial^{#2}_{\kern -.1em T}
  \fi
}

\def\Hdi#1#2{\icount=#1
  \ifnum\icount<1
  \widetilde H_{L}^{#2}\D
  \else
  \widetilde H_{R}^{#2}\D
  \fi
}

\begin{document}
\title[]{Numerical methods for time-fractional evolution equations with nonsmooth data: a concise overview}
\author[Bangti Jin]{Bangti Jin}
\address{Department of Computer Science, University College London, Gower Street, London, WC1E 2BT, UK.}
\email {b.jin@ucl.ac.uk,bangti.jin@gmail.com}

\author[Raytcho Lazarov]{$\,\,$Raytcho Lazarov$\,$}
\address{Department of Mathematics, Texas A\&M University, College Station, TX 77843, USA}
\email {lazarov@math.tamu.edu}

\author[Zhi Zhou]{$\,\,$Zhi Zhou$\,$}
\address{Department of Applied Mathematics,
The Hong Kong Polytechnic University, Hung Hom, Kowloon, Hong Kong, P.R. China}
\email {zhizhou@polyu.edu.hk}

\keywords{time-fractional evolution, subdiffusion, nonsmooth solution, finite element method, time-stepping, initial correction, error estimates, space-time formulation}

\date{\today}

\begin{abstract}
Over the past few decades, there has been substantial interest in evolution equations
that involving a fractional-order derivative of order $\alpha\in(0,1)$ in time, due to their
 many successful applications in engineering, physics, biology and finance. Thus,
it is of paramount importance to develop and to analyze efficient and
accurate numerical methods for reliably simulating such models, and the literature on
the topic is vast and fast growing. The present paper gives a concise overview on numerical
schemes for the subdiffusion model with nonsmooth problem data, which are important
for the numerical analysis of many problems arising in optimal control, inverse problems
and stochastic analysis. We focus on the following
aspects of the subdiffusion model: regularity theory, Galerkin finite element
discretization in space, time-stepping schemes (including convolution quadrature and L1 type
schemes), and space-time variational formulations, and compare the results with that for
standard parabolic problems. Further, these aspects are showcased with illustrative
numerical experiments and complemented with perspectives and pointers to relevant literature.
\end{abstract}

\maketitle

\section{Introduction}\label{sec:intro}
Diffusion is one of the most prominent transport mechanisms found in nature. The classical diffusion model
$\partial_t u-\Delta u=f$, which employs a first-order derivative $\partial_t u$ in time and the Laplace
operator $\Delta u$ in space, rests on the assumption that the particle motion is Brownian. One of the
distinct features of Brownian motion is a linear growth of the mean squared particle displacement with the
time $t$. Over the last few decades, a long list of experimental studies indicates that the Brownian motion
assumption may not be adequate for accurately describing some physical processes, and the mean squared
displacement can grow either sublinearly or superlinearly with time $t$, which are known as subdiffusion and
superdiffusion, respectively, in the literature. These experimental studies cover an extremely broad and
diverse range of important practical applications in engineering, physics, biology and finance, including electron
transport in Xerox photocopier \cite{ScherMontroll:1975}, visco-elastic materials \cite{Caputo:1967,
GinoaCerbelliRoman:1992}, thermal diffusion in fractal domains
\cite{Nigmatulin:1986}, column experiments \cite{HatanoHatano:1998} and protein transport in cell membrane
\cite{Kou:2008} etc. The underlying stochastic process for subdiffusion and superdiffusion is usually given
by continuous time random walk and L\'{e}vy process, respectively, and the corresponding macroscopic model for
the probability density function of the particle appearing at certain time instance $t$ and location $x$ is
given by a diffusion model with a fractional-order derivative in time and in space, respectively. We refer interested readers to the excellent
surveys \cite{MetzlerJeon:2014,MetzlerKlafter:2000} for an extensive list of practical applications and
physical modeling in engineering, physic, and biology and.

The present work surveys rigorous numerical methods for subdiffusion.
The prototypical mathematical model for subdiffusion is as follows. Let $\Omega\subset\mathbb{R}^d $
($d= 1,2,3$) be a convex polygonal domain with a boundary $\partial\Omega$, and consider the following
fractional-order parabolic problem for the function $u(x,t)$:
\begin{align}\label{eqn:pde}
\left\{\begin{aligned}
\Dal u(x,t) -\Delta u(x,t)  &= f(x,t) &&(x,t)\in\Omega\times(0,T], \\
u(x,t)&=0  &&(x,t)\in \partial\Omega\times(0,T], \\
u(x,0)&=v(x) &&x\in\Omega,
\end{aligned}
\right.
\end{align}
where $T>0$ is a fixed final time, $f \in L^\infty(0,T;L^2(\Omega))$ and $v\in L^2(\Omega)$ are given
source term and initial data, respectively, and $\Delta$ is the Laplace operator in space. Here
$\Dal u(t)$ denotes the Caputo fractional derivative in time $t$ of order $\alpha\in(0,1)$ \cite[p. 70]{KilbasSrivastavaTrujillo:2006}
\begin{align}\label{eqn:RLderive}
   \Dal u(t)= \frac{1}{\Gamma(1-\alpha)}\int_0^t(t-s)^{-\alpha} u'(s)\d s,
\end{align}
where $\Gamma(z)$ is the Gamma function defined by
\begin{equation*}
  \Gamma(z) = \int_0^\infty s^{z-1}e^{-s}\d s,\quad \Re (z)>0.
\end{equation*}
It is named after geophysicist Michele Caputo \cite{Caputo:1967}, who first introduced it
for describing the stress-strain relation in linear elasticity, although it was predated by
the work of Armenian mathematician Mkhitar Djrbashian \cite{Djrbashian:1993}. So more
precisely, it should be called Djrbashian-Caputo fractional derivative.
Note that the fractional derivative $\Dal u$ recovers the usual first-order derivative $u'(t)$ as
$\alpha\to1^-$, provided that $u$ is sufficiently smooth \cite[p. 100]{NakagawaSakamotoYamamoto:2010}.
Thus the model \eqref{eqn:pde} can be viewed as a fractional analogue of the classical parabolic
equation. Therefore, it is natural and instructive to compare its analytical and numerical properties with that of
standard parabolic problems.

\begin{remark}
All the discussions
below extend straightforwardly to a general second-order coercive and symmetric elliptic differential
operator, given by $\nabla\cdot(a(x)\nabla u(x)) - q(x)u(x)$ with $q\ge0$ a.e.
\end{remark}

Motivated by its tremendous success in the mathematical modeling of many physical problems, over the last
two decades there has been an explosive growth in the numerical methods, algorithms, and analysis of the subdiffusion model.
More recently this interest has been extended to related topics in optimal control, inverse problems and stochastic
fractional models. The literature on the topic is vast, and the list is still fast growing in the community of
scientific and engineering computation, and more recently also in the community of numerical analysis; see, e.g.,
the recent special issues on the topic at the journals \textit{Journal of Computational Physics}
\cite{KarniadakisHesthavenPodlubny:2015} and \textit{Computational Methods in Applied Mathematics}
\cite{JinLazarovVabishchevich:2017}, for some important progress in the area of numerical methods for fractional
evolution equations.

It is impossible to survey all important and relevant works in a short review.
Instead, in this paper, we aim at only reviewing relevant works on the numerical
methods for the subdiffusion model \eqref{eqn:pde} with nonsmooth problem data. This choice allows us to highlight some distinct
features common for many nonlocal models, especially how the smoothnes of the data influences the
solution and the corresponding numerical methods.
It is precisely these features that pose substantial new mathematical and computational challenges
when compared with standard parabolic problems, and extra care has to be exerted when developing
and analyzing numerical methods.
In particular, since the solution operators of the fractional model have limited smoothing property,
a numerical method  that requires high regularity of
the solution will impose severe restrictions (compatibility conditions) on the data and generally does
not work well and thus substantially limits its scope of potential applications. 
Finally, nonsmooth data analysis is fundamental to
the rigorous study of areas related to various applications, e.g.,  optimal control, inverse problems, and stochastic
fractional diffusion (see, e.g., \cite{JinLiZhou:2017control,JinRundell:2015,Yan:2005}).

Amongst the numerous possible choices, we shall focus the review on the following four aspects:
\begin{itemize}
  \item[(i)] Regularity theory in Sobolev spaces;
  \item[(ii)] Spatial discretization via finite element methods (FEMs), e.g., standard Galerkin, lumped mass and finite volume element methods;
  \item[(iii)] Temporal discretization via time-stepping schemes (including convolution quadrature and L1 type schemes);
  \item[(iv)] Space-time formulations (Galerkin or Petrov-Galerkin type).
\end{itemize}
In each aspect, we describe some representative results and leave most of technical proofs to the references.
Further, we compare the results with that for standard parabolic problems (see, e.g., \cite{Thomee:2006}) and
give some numerical illustrations of the theory.  Finally, we complement
each part with comments on future research problems and further references. The goal of the overview is
to give readers a flavor of the numerical analysis of nonlocal problems and potential pitfalls in developing
efficient numerical methods. We also refer the readers to the excellent
surveys for other nonlocal problems and applications, namely, on problem involving fractional
(spectral and integral)  Laplacian \cite{BonitoBorthagaray:2018}, on application to image processing \cite{yang2016fractional},
and on nonlocal problems arising in peridynamics \cite{DuGunzburgerLehoucqZhou:2012}. For a nice overview on
the numerical methods for fractional-order ordinary differential equations, we refer to the paper \cite{DiethelmFordFreedLuchko:2005}.

The rest of the paper is organized as follows. For the model \eqref{eqn:pde}, in
Section \ref{sec:reg}, we describe the regularity theory and
in Sections \ref{sec:fem} and \ref{sec:time-stepping}, we discuss the finite element methods and two
popular classes of time stepping schemes, i.e., convolution quadrature and L1 type schemes,  respectively.
Then, in Section \ref{sec:space-time} we discuss two
space-time formulations for problem \eqref{eqn:pde} with $v=0$. We conclude the overview with some
further discussions in Section \ref{sec:conclus}. Throughout, the discussions focus on the case of
nonsmooth problem data, and only references directly relevant are given. Obviously, the list of references is
not meant to be complete in any sense, and strongly biased by the personal taste and limited by the knowledge of the
authors. Throughout, the notation $c$ denotes a generic constant which may change at each occurrence, but it is
always independent of the discretization parameters $h$ and $\tau$ etc. In the paper we use the standard notation on
Sobolev spaces (see, e.g., \cite{AdamsFournier:2003}).

\section{Regularity of the solution}\label{sec:reg}
First, we describe some regularity results for the model \eqref{eqn:pde}, which are crucial for rigorous
numerical analysis. To this end, we need suitable function spaces. The most convenient one for our purpose
is the space $\dot H^s(\Omega)$ defined as below \cite[Chapter 3]{Thomee:2006}. Let $\{\lambda_j\}_{j=1}^\infty$
and $\{\varphi_j\}_{j=1}^\infty$ be respectively the eigenvalues (ordered nondecreasingly with multiplicity
counted) and the $L^2(\Omega)$-orthonormal eigenfunctions of the negative Laplace operator $-\Delta$ on the
domain $\Omega$ with a zero Dirichlet boundary condition. Then $\{\varphi_j\}_{j=1}^\infty$
forms an orthonormal basis in $L^2(\Omega)$.
For any real number $s\ge-1$, we denote by $\dH s$ the Hilbert space consisting of the
functions of the form
\begin{equation*}
  v = \sum_{j=1}^\infty \langle v,\varphi_j\rangle\varphi_j,
\end{equation*}
where $\langle\cdot,\cdot\rangle$ denotes the duality pairing between $H^{-1}(\Omega)$ and $H_0^1
(\Omega)$, and it coincides with the usual $L^2(\Omega)$ inner product $(\cdot,\cdot)$ if the function
$v\in L^2(\Omega)$. The induced norm $\|\cdot\|_{\dH s}$ is defined by
\begin{equation*}
  \|v\|_{\dH s}^2=\sum_{j=1}^{\infty}\lambda_j^s\langle v,\varphi_j \rangle^2.
\end{equation*}
Then, $\|v\|_{\dH 0}=\|v\|_{L^2(\Omega)}=(v,v)^\frac{1}{2}$ is the norm in $L^2(\Omega)$ and $\|v\|_{\dH {-1}}
= \|v\|_{H^{-1}(\Omega)}$ is the norm in $H^{-1}(\Omega)$. Besides, it is easy to verify that
$\|v\|_{\dH 1}= \|\nabla v\|_{L^2(\Omega)}$ is also an equivalent norm in $H_0^1(\Omega)$
and  $\|v\|_{\dH 2}=\|\Delta v\|_{L^2(\Omega)}$ is equivalent to the norm in $H^2(\Omega)\cap H^1_0(\Omega)$,
provided the domain $\Omega$ is convex \cite[Section 3.1]{Thomee:2006}. Note that the spaces $\dot H^s(\Omega)$, $s\ge -1$, form a
Hilbert scale of interpolation spaces. Motivated by this, we denote $\|\cdot\|_{H_0^s(\Omega)}$ to
be the norm on the interpolation scale between $H^1_0(\Omega)$ and $L^2(\Omega)$ when $s$
is in $[0,1]$ and $\|\cdot\|_{H_0^{s}(\Omega)}$ to be the norm on the interpolation scale between
$L^2(\Omega)$ and $H^{-1}(\Omega)$ when $s$ is in $[-1,0]$.  Then, $\| \cdot \|_{H_0^s(\Omega)}$
and $\|\cdot\|_{\dH s}$ are equivalent for $s\in [-1,0]$ by interpolation.

There are several different ways to analyze problem \eqref{eqn:pde}. We outline one approach to derive
regularity results by means of Laplace transform below. We denote the Laplace transform of
a function $f:(0,\infty)\to\mathbb{R}$ by $\widehat{f}$ below. The starting point of the analysis is the following
identity on the Laplace transform of the Caputo fractional derivative $\partial_t^\alpha u(t)$
\cite[Lemma 2.24, p. 98]{KilbasSrivastavaTrujillo:2006}
\begin{equation*}
   \widehat{\partial_t^\alpha u}(z) = z^\alpha \widehat{u}(z) - z^{\alpha-1}u(0).
\end{equation*}
By viewing $u(t)$ as a vector-valued function, applying Laplace transform to both sizes of \eqref{eqn:pde} yields
\begin{equation*}
   z^\alpha \widehat{u}(z) -\Delta \widehat{u} = \widehat f + z^{\alpha-1}u(0),
\end{equation*}
i.e.,
\begin{equation*}
  \widehat{u}(z) = (z^\alpha-\Delta)^{-1}(\widehat f + z^{\alpha-1}u(0)).
\end{equation*}
By inverse Laplace transform and the convolution rule, the solution $u(t)$ can be formally represented by
\begin{align}\label{eqn:Sol-expr-u-const}
u(t)= F(t)v + \int_0^t E(t-s) f(s) \d s ,
\end{align}
where the solution operators $F(t)$ and $E(t)$ are respectively defined by
\begin{align*}
F(t):=\frac{1}{2\pi {\rm i}}\int_{\Gamma_{\theta,\delta }}e^{zt} z^{\alpha-1} (z^\alpha-\Delta  )^{-1}\, \d z \quad\mbox{and}\quad
E(t):=\frac{1}{2\pi {\rm i}}\int_{\Gamma_{\theta,\delta}}e^{zt}  (z^\alpha-\Delta)^{-1}\, \d z ,
\end{align*}
with integration over a contour $\Gamma_{\theta,\delta}$ in the complex plane
(oriented counterclockwise), i.e.,
\begin{equation*}
  \Gamma_{\theta,\delta}=\left\{z\in \mathbb{C}: |z|=\delta, |\arg z|\le \theta\right\}\cup
  \{z\in \mathbb{C}: z=\rho e^{\pm\mathrm{i}\theta}, \rho\ge \delta\} .
\end{equation*}
Throughout, we fix $\theta \in(\frac{\pi}{2},\pi)$ so that $z^{\al} \in \Sigma_{\al\theta}
\subset \Sigma_{\theta}:=\{0\neq z\in\mathbb{C}: {\rm arg}(z)\leq\theta\},$ for all $z\in\Sigma_{\theta}$.
Recall the following resolvent estimate for the  Laplacian $\Delta$ with homogenous Dirichlet boundary condition \cite{}:
\begin{equation} \label{eqn:resol}
  \| (z-\Delta )^{-1} \|\le c_\phi |z|^{-1},  \quad \forall z \in \Sigma_{\phi},
  \,\,\,\forall\,\phi\in(0,\pi),
\end{equation}
where $\|\cdot\|$ denotes the operator norm from $L^2(\Omega)$ to $L^2(\Omega)$.

Equivalently, using the eigenfunction expansion $\{(\lambda_j,\varphi_j)\}_{j=1}^\infty$, these
operators can be expressed as
\begin{equation*}
  F(t)v = \sum_{j=1}^\infty E_{\alpha,1}(-\lambda_jt^\alpha)(v,\varphi_j)\varphi_j\quad \mbox{and}\quad
  E(t)v = \sum_{j=1}^\infty t^{\alpha-1}E_{\alpha,\alpha}(-\lambda_jt^\alpha)(v,\varphi_j)\varphi_j.
\end{equation*}
Here $E_{\alpha,\beta}(z)$ is the two-parameter Mittag-Leffler function defined by
\cite[Section 1.8, pp. 40-45]{KilbasSrivastavaTrujillo:2006}
\begin{equation*}
  E_{\alpha,\beta}(z) = \sum_{k=0}^\infty \frac{z^{k}}{\Gamma(k\alpha+\beta)}\quad \forall z\in \mathbb{C}.
\end{equation*}
The Mittag-Leffler function $E_{\alpha,\beta}(z)$ is a generalization of the familiar exponential function
$e^z$ appearing in normal diffusion, and it can be evaluated efficiently via contour integral \cite{GorenfloLoutchko:2002b,
SeyboldHilfer:2008}. Since the solution operators involve only $E_{\alpha,\beta}(z)$
with $z$ being a negative real argument, the following decay behavior $E_{\alpha,\beta}(z)$ is crucial
to the smoothing properties of $F(t)$ and $E(t)$: for any $\alpha\in (0,1)$, the function $E_{\alpha,1}
(-\lambda t^\alpha)$ decays only polynomially like $t^{-\alpha}$ as $t\to\infty$ \cite[equation (1.8.28),
p. 43]{KilbasSrivastavaTrujillo:2006}, which contrasts sharply with the exponential decay for $e^{-\lambda t}$
appearing in normal diffusion. These important features directly translate into the limited smoothing
property in space and in time for the solution operators $E(t)$ and $F(t)$.

Next, we state a few regularity results. The proof of these results can be
found in, e.g., \cite{Bajlekov:2001,JinLiZhou:nonlinear,SakamotoYamamoto:2011}.
\begin{theorem} \label{thm:reg-u}
Let $u(t)$ be the solution to problem \eqref{eqn:pde}. Then the following statements hold.
\begin{itemize}
  \item[(i)] If $v \in \dH s$ with $s\in(-1,2]$ and $f=0$, then $u(t)\in \dH {s+2}$ and
\begin{equation*}
\|  \partial_t^{(m)}  u(t) \|_{\dH p} \le c t^{\frac{(s-p)\alpha}{2}-m} \| v \|_{\dH s}
\end{equation*}
with $0\le p-s\le 2$ and any integer $m\ge 0$ .
  \item[(ii)]If $v=0$ and $f\in L^p(0,T;L^2(\Omega))$ with $1<p<\infty$, then there holds
\begin{equation*}
\|  u\|_{L^p(0,T;\dot H^2(\Omega))}
+\|\Dal u\|_{L^p(0,T;L^2(\Omega))}
\le c\|f\|_{L^p(0,T;L^2(\Omega))}.
\end{equation*}
Moreover, if $f\in L^\infty(0,T;L^2(\Omega))$, we have for any $\epsilon\in(0,1)$
\begin{equation*}
\| u(t) \|_{ \dot H^{2-\epsilon}(\Omega) }  \le c \epsilon^{-1} t^{\epsilon\alpha} \|  f \|_{L^\infty(0,t;L^2(\Omega))}.
\end{equation*}
  \item[(iii)]If $v=0$ and $f\in C^{m-1}([0,T];L^2(\Omega))$, $\int_0^t (t-s)^{\alpha-1} \|  \partial_s^{(m)} f(s) \|_{ L^2(\Omega)}\d s<\infty$, then there holds
\begin{equation*}
\| \partial_t^{(m)}  u(t) \|_{ L^2(\Omega) }  \le c   \sum_{k=0}^{m-1} \| \partial_t^{(k)}f (0) \|_{L^2(\Omega)}t^{\alpha-k}+ \int_0^t (t-s)^{\alpha-1} \|  \partial_s^{(m)} f(s) \|_{ L^2(\Omega)}\d s.
\end{equation*}
\end{itemize}
\end{theorem}

The estimate in Theorem \ref{thm:reg-u}(i) indicates that for homogeneous problems, the solution $u(t)$ is smooth
in time $t>0$ (actually analytic in a sector in the complex plane $\mathbb{C}$ \cite[Theorem 2.1]{SakamotoYamamoto:2011}),
but has a weak singularity around $t=0$. The strength of the singularity depends on the regularity of the initial data
$v$: the smoother is $v$ (measured in the space $\dot H^s(\Omega)$), the less singular is the solution $u$ at the initial
layer. Interestingly, even if the initial data $v$ is very smooth, the solution $u$ is generally not very smooth in
time in the fractional case, which also differs from the standard parabolic case. By now, it is well known that smooth
solutions are produced by a small class of data \cite{Stynes:2016}. The condition $0\leq p-s\leq 2$ in Theorem \ref{thm:reg-u}(i)
represents an essential restriction on the smoothing property in space of order two. This restriction contrasts sharply
with that for the standard diffusion equation: the following estimate
\begin{equation*}
  \|  \partial_t^{(m)}  u(t) \|_{\dH p} \le c t^{\frac{s-p}{2}-m} \| v \|_{\dH s}
\end{equation*}
holds for any $t>0$ and any $p\geq s, m\geq 0$ (see, e.g.  \cite[Lemma 3.2, p. 39]{Thomee:2006}).
This means that  the solution operator for standard parabolic problems is infinitely smoothing in space,
as long as $t>0$. The limited smoothing property in space of the model \eqref{eqn:pde}
represents one very distinct feature, which is generic for many other nonlocal (in time) models.

The first inequality in Theorem \ref{thm:reg-u}(ii) is often known as maximal $L^p$ regularity,
which is very useful in the numerical analysis of nonlinear problems (see, e.g., \cite{KovacsLiLubich:2016,
AkrivisLiLubich2017} for standard parabolic problems and \cite{JinLiZhou:nonlinear} for subdiffusion). Theorem
\ref{thm:reg-u}(iii) asserts that the temporal regularity of the solution $u(t)$ is essentially determined by that
of the right hand side $f$. The solution $u(t)$ can still have weak singularity near $t=0$, even for a very
smooth source term $f$, which differs dramatically again from standard parabolic problems. In order to have high temporal regularity
uniformly in time $t$, for the homogeneous problem, it is necessary to impose the following (rather restrictive)
compatibility conditions: $\partial_t^{(k)}f(0)=0$, $k=0,\ldots,m-1$. In the numerical analysis, it is important to
take into account the initial singularity of the solution,
which represents one of the main challenges in developing robust numerical methods.

Now we illustrate the results for the homogeneous problem.
\begin{example}\label{exam:reg}
Consider problem \eqref{eqn:pde} on the unit interval $\Omega=(0,1)$ with
\begin{itemize}
 \item[(i)] $v=\sin(\pi x)$ and $f=0$;
 \item[(ii)] $v=\delta_{0.5}(x)$, with $\delta_{0.5}(x)$ the Dirac $\delta$ function concentrated at $x=0.5$, and $f=0$.
\end{itemize}
The solution $u(t)$ in case {\rm(i)} is given by $u(t)=E_{\alpha,1}(-\pi^2t^\alpha)\sin(\pi x)$. Since $\sin(\pi t)$
is a Dirichlet eigenfunction of the negative Laplacian $-\Delta$ on $\Omega$, it
is easy to see that for any $s\geq0$, $v\in \dH s$, but the solution $u(t)$ has limited temporal
regularity for any $\alpha\in(0,1)$: as $t\to0$, $E_{\alpha}(-\pi^2t^\alpha)\sim 1-\frac{\pi^2}{\Gamma(\alpha+1)}t^\alpha$,
which is continuous at $t=0$ but with an unbounded first-order derivative. This observation clearly reflects the inherently limited smoothing property in time of
problem \eqref{eqn:pde}. It contrasts sharply with the standard parabolic case, $\alpha=1$, for
which the solution $u(t)$ is given explicitly by $u(t)=e^{-\pi^2 t}\sin (\pi x)$ and is $C^\infty $ in time.
In Fig. \ref{fig:diffwave2-space}, we show the solution profiles for $\alpha=0.5$ and
$\alpha=1$ at two different time instances for case {\rm(ii)}. Observe that the solution
profile for $\alpha=1$ decays much faster than that for $\alpha=0.5$. For any $t>0$, the solution
$u(t)$ is very smooth in space for $\alpha=1$, but it remains nonsmooth for $\alpha=0.5$.
In the latter case, the kink at $x=0.5$ in the plot shows clearly the limited spatial smoothing
property of the solution operator $F(t)$, and it remain no matter how long the problem evolves.
\end{example}

\begin{figure}[h]
\subfigure[$t=0.15$]{\includegraphics[trim = .1cm .1cm .2cm 0.5cm, clip=true,width=0.3\textwidth]{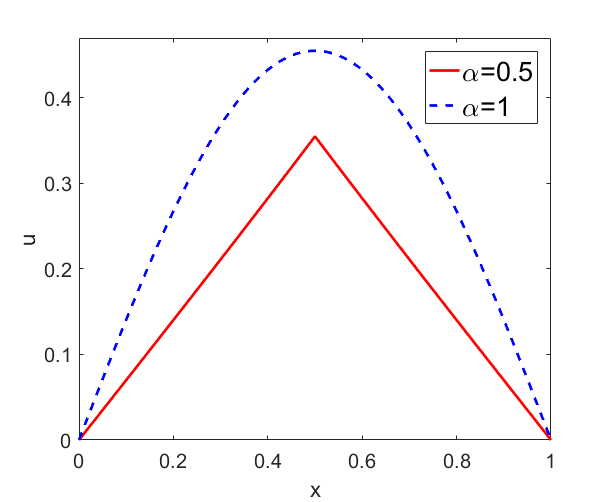}}
\subfigure[$t=0.2$]{\includegraphics[trim = .1cm .1cm .2cm 0.5cm, clip=true,width=0.3\textwidth]{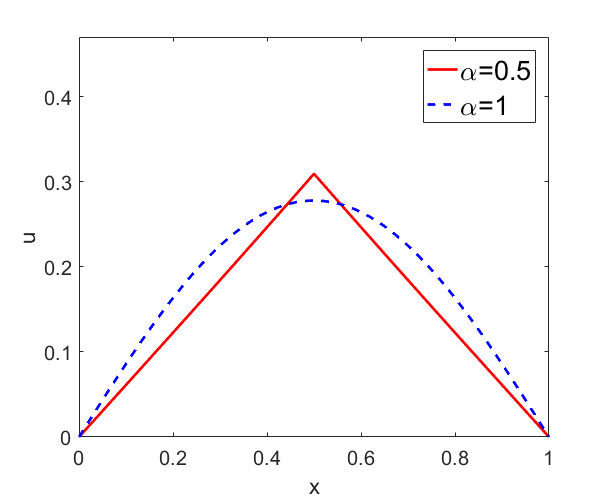}}
\subfigure[$t=0.3$]{\includegraphics[trim = .1cm .1cm .2cm 0.5cm, clip=true,width=0.3\textwidth]{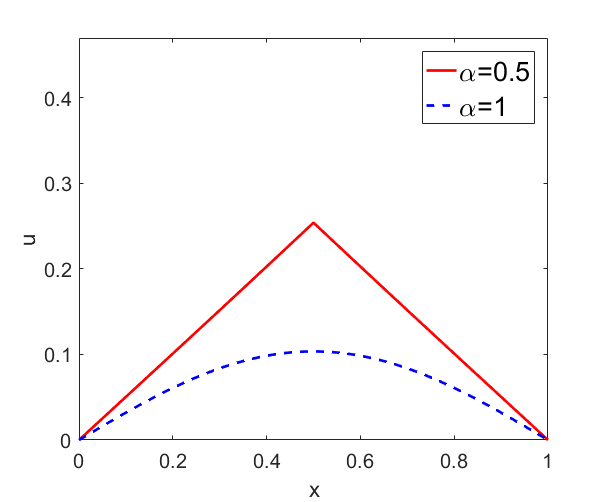}}
\caption{The solution profiles for Example \ref{exam:reg}(ii) at three time instances for $\alpha=0.5$ and $1$.}
\label{fig:diffwave2-space}
\end{figure}


The analytical theory of problem \eqref{eqn:pde} has been developed successfully in the last two decades, e.g.,
\cite{Bajlekov:2001,EidelmanKochubei:2004,Luchko:2009,Pskhu:2009,McLean:2010,SakamotoYamamoto:2011,Kochubei:2014,JinLiZhou:nonlinear,
AllenCaffarelliVasseur:2016,LiuRundellYamamoto:2016,Yamamoto:2018,GalWarma:2017}; see also the monograph \cite{Pruss:1993} for closely
related evolutionary integral equations.
Eidelman and Kochubei \cite{EidelmanKochubei:2004} derived fundamental solutions
to problem in the whole space using Fox $H$-functions, and derived various estimates,
see also \cite{SchneiderWyss:1989,GorenfloLuchkoYamamoto:2015};. Luchko \cite{Luchko:2009} studied
the existence and uniqueness of a strong solution. Sakamoto and Yamamoto \cite{SakamotoYamamoto:2011}
analyzed the problem by means of separation of variables, reducing it to an infinite system of
fractional-order ODEs, studied the existence and uniqueness of weak solutions, and proved
various regularity results including the asymptotic behavior of the solution
for $t \to 0$ and $t \to \infty$. We note that the Laplace transform technique described above is
essentially along the same line of reasoning. The important issue of properly interpreting
the initial condition (for $\alpha$ close to zero) was discussed in \cite{GorenfloLuchkoYamamoto:2015,LiLiu:2016}.

It is worth noting that techniques like separation of variables and Laplace transform are
most convenient for analyzing time-independent elliptic operators. For time-dependent elliptic
operators or nonlinear problems, e.g., time-dependent diffusion coefficients and Fokker-Planck
equation, energy arguments \cite{VergaraZacher:2015} or
perturbation arguments \cite{KimKimLim:2017} can be used to show existence and uniqueness of the solution.
However, the slightly more refined stability estimates, needed for numerical analysis of nonsmooth problem data,
often do not directly follow and have to be derived separately. This represents one of the main obstacles
in extending the results below for the model problem \eqref{eqn:pde} to these important classes of applied problems.

\section{Spatially semidiscrete approximation}\label{sec:fem}
Now we describe several spatially semidiscrete finite element schemes for problem \eqref{eqn:pde}
using the standard notation from the classical monograph \cite{Thomee:2006}. Semidiscrete methods
are usually not directly implementable and used in practical computations, but they are important
for understanding the role of the regularity of problem data and also for the analysis of some
space-time formulations and spectral, Pad\'e, and rational approximations.
Let ${\{\mathcal{T}_h\}}_{0<h<1}$ be a family
of shape regular and quasi-uniform partitions of the domain $\Omega$ into $d$-simplexes, called
finite elements, with the mesh size $h$ denoting the maximum diameter of the elements.
An approximate solution $u_h$ is then sought in the finite element space $X_h\equiv
X_h(\Omega)$ of continuous piecewise linear functions over the triangulation $\mathcal{T}_h $, defined by
\begin{equation*}
  X_h =\left\{\chi\in H^1_0(\Omega): \ \chi ~~\mbox{is a linear function over}  ~~\K,
 \,\,\,\,\forall \K \in \mathcal{T}_h\right\}.
\end{equation*}
To describe the schemes, we need the $L^2(\Omega)$ projection $P_h:L^2(\Omega)\to X_h$ and
Ritz projection $R_h:\dH1\to X_h$, respectively, defined by (recall that $(\cdot, \cdot)$
denotes the $L^2(\Omega)$ inner product)
\begin{equation*}
  \begin{aligned}
    (P_h \psi,\chi) & =(\psi,\chi) \quad\forall \chi\in X_h,\psi\in L^2(\Omega),\\
    (\nabla R_h \psi,\nabla\chi) & =(\nabla \psi,\nabla\chi) \quad \forall \chi\in X_h, \psi\in \dot H^1(\Omega).
  \end{aligned}
\end{equation*}
Then by means of duality, the operator $P_h$ can be boundedly extended to $\dH s$, $s \in [-1,0]$.
The following approximation properties of $R_h$ and $P_h$ are well known:
\begin{align*}
  \|P_h\psi-\psi\|_{L^2(\Omega)}+h\|\nabla(P_h\psi-\psi)\|_{L^2(\Omega)}& \leq ch^q\|\psi\|_{H^q(\Omega)}\quad \forall\psi\in \dH q, q=1,2,\\
  \|R_h\psi-\psi\|_{L^2(\Omega)}+h\|\nabla(R_h\psi-\psi)\|_{L^2(\Omega)}& \leq ch^q\|\psi\|_{H^q(\Omega)}\quad \forall\psi\in \dH q, q=1,2.
\end{align*}

By multiplying both sides of equation \eqref{eqn:pde} by a test function $\varphi\in H_0^1(\Omega)$,
integrating over the domain $\Omega$ and then applying integration by parts formula yield the following
weak formulation of problem \eqref{eqn:pde}: find $u(t) \in H^1_0(\Omega)$ for $t>0$ such that
\begin{equation}\label{eqn:weak}
(\Dal u(t), \varphi) + a(u(t), \varphi) = (f,\varphi),
\quad \forall \varphi \in H^1_0(\Omega),\ \ t>0, \mbox{with }u(0)=v,
\end{equation}
where $a(u,\varphi)=(\nabla u, \nabla \varphi)$ for $u,\varphi\in H_0^1(\Omega)$ denotes the bilinear form
for the elliptic operator $A=-\Delta$ (with a zero Dirichlet boundary condition). Then the spatially
semidiscrete approximation of problem \eqref{eqn:pde} is to find $u_h (t)\in X_h$ such that
\begin{equation}\label{eqn:semi}
 {[ \Dal u_{h}(t),\chi]}+ a(u_h(t),\chi)= (f,\chi),
\quad \forall \chi\in X_h,\ t >0, \mbox{with }u_h(0)=v_h,
\end{equation}
where $v_h \in X_h$ is an approximation of the initial data $v$, and the notation $[\cdot,\cdot]$ refers to a suitable
inner product on the space $X_h$, approximating the usual $L^2(\Omega)$ inner product $(\cdot,\cdot)$.
Following Thom\'ee \cite{Thomee:2006}, we shall take $v_h=R_hv$ in case of smooth initial data $v\in \dH2$ and $v_h=P_hv$
in case of nonsmooth initial data, i.e., $v\in \dH s$, $-1\leq s\le0$. Moreover, the spatially semidiscrete variational
problem \eqref{eqn:semi} can be written in an operator form as
\begin{equation*}
 \Dal u_{h}(t) +  A_h u_h(t) = f_h, \quad \forall \chi\in X_h,\ t >0, \mbox{with }u_h(0)=v_h,
\end{equation*}
where $A_h$ is a discrete approximation to the elliptic operator $A$ 
on the space $X_h$, and will be given below.

Based on the abstract form \eqref{eqn:semi}, we shall present three predominant finite element type
discretization methods in space, i.e., standard Galerkin finite element (SG) method,
lumped mass (LM) method and finite volume element (FVE) method, below. In passing, we note that
in principle any other spatial discretization methods, e.g., finite difference methods \cite{LinXu:2007,ZhangSun:2011adi},
collocation, and spectral methods \cite{LinXu:2007,chenHestaven2015multi} can also be used. Our choice of the FEMs is motivated
by nonsmooth problem data.

\subsection{Standard Galerkin finite element.}
The SG method is obtained from \eqref{eqn:semi}
when the approximate inner product $[\cdot,\cdot]$ is chosen to be the usual $L^2(\Omega)$ inner product $(\cdot,\cdot)$.
The SG method was first developed and rigorously analyzed for nonsmooth data in \cite{JinLazarovZhou:SIAM2013,JinLazarovPasciakZhou:2013,
JinLazarovPasciakZhou:2015} for problem \eqref{eqn:pde} on convex polygonal domains, and in
\cite{LeMcLeanLamichhane:2017} for the case of nonconvex domains.

Upon introducing the discrete Laplacian $\Delta_h: X_h\to X_h$ defined by
\begin{equation*}
  -(\Delta_h\psi,\chi)=(\nabla\psi,\nabla\chi)\quad\forall\psi,\,\chi\in X_h,
\end{equation*}
and $f_h= P_h f$, we may write the spatially discrete problem \eqref{eqn:semi} as
\begin{equation}\label{fem-operator}
   \Dal u_{h}(t)-\Delta_h u_h(t) =f_h(t) \for t\ge0 \quad \mbox{with} \quad  u_h(0)=v_h.
\end{equation}
Now we introduce the semidiscrete analogues of $F(t)$ and $E(t)$ for $t>0 $:
\begin{align*}
F_h(t):=\frac{1}{2\pi {\rm i}}\int_{\Gamma_{\theta,\delta }}e^{zt} z^{\alpha-1} (z^\alpha- \Delta_h)^{-1}\, \d z \qquad\text{and}\qquad
E_h(t):=\frac{1}{2\pi {\rm i}}\int_{\Gamma_{\theta,\delta}}e^{zt}  (z^\alpha- \Delta_h)^{-1}\, \d z.
\end{align*}
Then the solution $u_h(t)$ of the semidiscrete problem \eqref{fem-operator} can be succinctly expressed by:
\begin{equation}\label{Duhamel_o}
     u_h(t)= F_h(t) v_h + \int_0^t E_h(t-s) f_h(s)\,\d s.
\end{equation}

Now we give pointwise-in-time $L^2(\Omega)$ error estimates for the semidiscrete Galerkin approximation $u_h$.
\begin{theorem}\label{thm:error-fem}
Let $u$ be the solution of problem \eqref{eqn:pde} and $u_h$ be the solution of problem \eqref{fem-operator},
respectively. Then with $\ell_h=|\log h|$, for any $t>0$, the following error estimates hold:
\begin{itemize}
\item[$\rm(i)$] If $f\equiv 0$, $v\in L^2(\Omega)$, and $v_h=P_hv$, then
\begin{equation*}
 \|(u-u_h)(t)\|_{L^2(\Omega)} \le ch^2\ell_h t^{-\alpha} \|v\|_{L^2(\Omega)}.
\end{equation*}
\item[$\rm(ii)$] If $f\equiv0$, $v\in \dH 2$, $v_h=R_hv$, then
\begin{equation*}
 \|(u-u_h)(t)\|_{L^2(\Omega)}  \le ch^2 \| \Delta v\|_{L^2(\Omega)}.
\end{equation*}
\item[$\rm(iii)$] If $f\in L^\infty(0,T;L^2(\Omega))$ and $v\equiv0$, then
\begin{equation*}
 \|(u-u_h)(t)\|_{L^2(\Omega)}  \le c h^2 \ell_h^2\| f\|_{L^\infty(0,t;L^2(\Omega))}.
\end{equation*}
\end{itemize}
\end{theorem}
\begin{proof}
We only briefly sketch the proof for part (i) to give a flavor, and refer interested
readers to \cite{JinLazarovZhou:SIAM2013,JinLazarovPasciakZhou:2015} for further details.
In a customary way, we split the error $u_h(t)-u(t)$ into two terms as
\begin{equation*}
u_h-u= (u_h-P_hu)+(P_hu-u):=\vth + \rh.
\end{equation*}
By the approximation property of the $L^2(\Omega)$ projection $P_h$ and Theorem \ref{thm:reg-u}, we have for any $t>0$
\begin{equation*}
 \| \rh(t) \|_{L^2(\Omega)} 
 \le ch^2 \| u(t) \|_{H^2(\Omega)} \le ch^2 t^{-\alpha} \|  v \|_{L^2(\Omega)}.
\end{equation*}
So it remains to obtain proper estimates on $\vth(t) $. Obviously, $ P_h \Dal \rh = \Dal P_h(P_hu-u)=0$
and using the identity $\Delta_hR_h=P_h\Delta$ \cite[equation (1.34), p. 11]{Thomee:2006}, we deduce that $\vth$ satisfies:
\begin{equation*}
 \Dal \vth(t) -\Delta_h \vth(t) = -\Delta_h (R_h u - P_h u)(t), \quad t>0, \quad \vth(0)=0.
\end{equation*}
Then with the help of Duhamel's formula \eqref{Duhamel_o}, $\vth(t)$ can be represented by
\begin{equation*}
   \begin{aligned}
    \vth(t) &=  -\int_0^t E_h(t-s)\Delta_h(R_hu-P_hu)(s)\,\d s\\
     &= \int_0^t (-\Delta_h)^{1-\epsilon}E_h(t-s)(-\Delta_h)^{\epsilon}(R_hu-P_hu)(s)\,\d s,
   \end{aligned}
\end{equation*}
where the constant $\epsilon\in(0,1)$ is to be chosen below. Consequently,
\begin{equation*}
    \|\vth(t)\|_{L^2(\Omega)} \leq
\int_0^t\|(-\Delta_h)^{1-\epsilon}E_h(t-s)\|\|(-\Delta_h)^{\epsilon}(R_hu-P_hu)(s)\|_{L^2(\Omega)}\d s,
\end{equation*}
where $(-\Delta_h)^\epsilon$ is the fractional power of $-\Delta_h$ defined in the spectral sense.
That is, if $(\lambda_j^h, \phi_j^h)$ are the eigenvalues and eigenfunctions of $-\Delta_h$,
then for $v \in X_h$,  $(-\Delta_h)^\epsilon v = \sum_j  (\lambda_j^h)^\epsilon (v,\phi_j^h) \phi_j^h$.
Now recall the smoothing property of the semidiscrete solution operator $E_h(t)$
\begin{equation*}
  \|E_h(t)(-\Delta_h)^s\|\leq ct^{-1+(1-s)\alpha}\quad \forall s\in[0,1],
\end{equation*}
which follows directly from the resolvent estimate \eqref{eqn:resol} (for $\Delta_h$), and the inverse estimate for FEM functions
\begin{equation*}
  \|(-\Delta_h)^s(R_hu-P_hu)(t)\|_{L^2(\Omega)} \leq c h^{-2s}\|(R_hu-P_hu)(t)\|_{L^2(\Omega)}\quad\forall s\in[0,1].
\end{equation*}
Thus, by the triangle inequality and the approximation properties of $R_h$ and $P_h$, we deduce
\begin{equation*}
  \|(R_hu-P_hu)(t)\|_{L^2(\Omega)} \leq \|(R_hu-u)(t)\|_{L^2(\Omega)} + \|(P_hu-u)(t)\|_{L^2(\Omega)}\leq ch^2\|u(t)\|_{H^2(\Omega)}.
\end{equation*}
The preceding estimates together with Theorem \ref{thm:reg-u} imply
\begin{equation*}
\begin{split}
  \|\vth(t) \|_{L^2(\Omega)}  &\le c h^{2-2\epsilon}\int_0^t  (t-s)^{\epsilon\alpha-1}\|u(s)\|_{H^2(\Omega)}\,\d s\\
  &\le c h^{2-2\epsilon}\| v \|_{L^2(\Omega)}\int_0^t  (t-s)^{\epsilon\alpha-1}s^{-\alpha}\,\d s\\
  & \le c \epsilon^{-1}h^{2-2\epsilon} t^{-\alpha} \| v \|_{L^2(\Omega)}.
\end{split}
\end{equation*}
The desired assertion follows by choosing $\epsilon=1/\ell_h$.
\end{proof}

\begin{remark}\label{rmk:fem-err}
It is instructive to compare the error estimate in Theorem \ref{thm:error-fem} with that for standard parabolic problems. For example, in
the latter case, for the homogeneous problem with $v\in L^2(\Omega)$, the following error estimate holds \cite[Theorem 3.5, p. 47]{Thomee:2006}:
\begin{equation*}
  \|u(t)-u_h(t)\|_{L^2(\Omega)}\leq ch^2t^{-1} \|v\|_{L^2(\Omega)}.
\end{equation*}
This estimate is comparable with that in Theorem \ref{thm:error-fem}(i), apart from the log factor $\ell_h$,
which can be overcome using an operator trick due to Fujita and Suzuki \cite{FujitaSuzuki:1991}. Hence,
in the limit $\alpha\to1^-$, the result in the fractional case essentially recovers that for the standard parabolic case.
The log factor $\ell_h$ in the estimate for the inhomogeneous problem in Theorem \ref{thm:error-fem}(iii)
is due to the limited smoothing property, cf. Theorem \ref{thm:reg-u}(ii). It is unclear whether the factor
$\ell_h$ is intrinsic or due to the limitation of the proof technique.

Upon extension, the following error estimate analogous to Theorem \ref{thm:error-fem}(i) holds for very weak
initial data, i.e., $v\in H^{-s}(\Omega)$, $0\leq s\leq 1$ \cite[Theorem 2]{JinLazarovPasciakZhou:2013}:
\begin{equation*}
  \|u(t)-u_h(t)\|_{L^2(\Omega)}\leq ch^{2-s}\ell_ht^{-\alpha}\|v\|_{H^{-s}(\Omega)}.
\end{equation*}
\end{remark}

\subsection{Two variants (lumped mass and finite volume) of Galerkin method.}
Now we discuss two variants of the standard Galerkin FEM, i.e., lumped mass FEM and finite volume element method.
These methods have also been analyzed for nonsmooth data, but less extensively \cite{JinLazarovZhou:SIAM2013,
JinLazarovPasciakZhou:2015,KaraaMustaphaPani:2017,Kopteva:2017}. These variants are essential for some applications:
the lumped mass FEM is important for preserving qualitative properties of the approximations, e.g., positivity
\cite{ChatzipantelidisHorvathThomee:2015,JinLazarovThomeeZhou:2017}, while the finite volume method
inherits the local conservation property of the physical problem.

First, we describe the lumped mass FEM  (see, e.g. \cite[Chapter 15, pp. 239--244]{Thomee:2006}), where the mass
matrix is replaced by a diagonal matrix with the row sums of the original mass matrix as its diagonal elements.
 Specifically, let $\zK_j $, $j=1,\dots,d+1$ be
the vertices of a $d$-simplex $\K \in \mathcal{T}_h$. Consider the following quadrature formula
\begin{equation*}
Q_{\K,h}(f) = \frac{|\K|}{d+1} \sum_{j=1}^{d+1} f(\zK_j) \approx \int_\K f \d x\quad \forall f\in C(K).
\end{equation*}
where $|K|$ denotes the area/volume of the simplex $K$. Then we define an approximate
$L^2(\Omega)$-inner product $(\cdot,\cdot)_h$ in $X_h$ by
\begin{equation*}
(w, \chi)_h = \sum_{\K \in \mathcal{T}_h}  Q_{\K,h}(w \chi).
\end{equation*}
The lumped mass FEM is to find $ \luh (t)\in X_h$ such that
\begin{equation*}
 {(\Dal \luh, \chi)_h}+ a(\luh,\chi)= (f, \chi)
\quad \forall \chi\in X_h,\ t >0, \quad\mbox{with }\luh(0)=P_hv.
\end{equation*}
Then we introduce the discrete Laplacian $-\bDelh:X_h\rightarrow X_h$, corresponding to
the inner product $(\cdot,\cdot)_h$, by
\begin{equation*}
  -(\bDelh\psi,\chi)_h = (\nabla \psi,\nabla \chi)\quad \forall\psi,\chi\in X_h.
\end{equation*}

\begin{remark}
In a rectangular domain $\Omega$ and a uniform square mesh partitioned into triangles
{\rm(}by connecting the lower left corner with the upper right corner{\rm)} the operator
${\color{blue}\bar\Delta_h}$ is identical with the canonical five-point finite difference approximation of
the Laplace operator. Such relation may allow extending the analysis below to various
finite difference approximations of problem \eqref{eqn:pde}.
\end{remark}

Also, we introduce a projection operator $\bar P_h: L^2(\Omega) \rightarrow X_h$ by
$$(\bar P_h f, \chi)_h = (f, \chi), \quad \forall \chi\in X_h.$$
Then with $f_h = \bar P_hf$, the lumped mass FEM can be written in an operator form as
\begin{equation}\label{eqn:fem-lumped}
  \Dal{\luh}(t)-\bDelh \luh(t) = f_h(t) \quad \mbox{ for }t\geq 0 \quad \mbox{with }\luh(0)= P_h v.
\end{equation}

Next, we describe the finite volume element (FVE) method (see, e.g., \cite{ChouLi:2000,ChatzLazarovThomee:2013}).
It is based on a discrete version of the local conservation law
\begin{equation}
	\label{FVE}
  \int_V\Dal u(t) \d x - \int_{\partial V}\frac{\partial u}{\partial n}\d s
  = \int_V f \,\d x, \quad \mbox{for }t> 0,
\end{equation}
valid for any $V\subset\Omega$ with a piecewise smooth boundary $\partial V$, with $n$ being the unit outward
normal to $\partial V$.  The FVE requires \eqref{FVE} to be satisfied for $V=V_j,\ j=1,
\dots,N$, which  are disjoint and known as control volumes associated with the nodes $P_j$ of
$\mathcal{T}_h$. Then the discrete problem reads: find $\tu_{h}(t)\in X_h$ such that
\begin{equation}\label{FVE-0}
    \int_{V_j}\Dal \tu_{h}(t) \,\d x-\int_{\partial V}\frac{\tu_{h}}{\partial n}\d s =\int_{V_j} f \,\d x,
    \quad   \mbox{for }t\geq 0,
    \quad \text{with}~~ \tu_h(0)=P_hv,
\end{equation}
It can be recast as a Galerkin method \cite{ChatzLazarovThomee:2013}, by letting
\begin{equation*}
Y_h= \left\{ \varphi\in L^2(\Omega): \varphi|_{V_j} = \text{constant},
~~j=1,2,...,N; \ \varphi = 0~~\text{outside}~~ \cup_{j=1}^N V_j  \right\},
\end{equation*}
introducing the interpolation operator $J_h:C(\Omega)\to Y_h$ by
$(J_hv)(P_j)=v(P_j)$, $j=1,\ldots,N$, and
then defining an approximate $L^2(\Omega)$ inner product
$\langle\chi,\psi\rangle=(\chi,J_h\psi)$ for all $\chi,\psi
\in X_h$. The FVE method \eqref{FVE-0} can be reformulated by
\begin{equation*}
 \langle\Dal\tu_{h}(t),\chi \rangle + a(\tu(t),\chi) = (f(t), J_h \chi) \quad \mbox{ for }t\geq 0 \quad \mbox{with }\tu_h(0)= P_h v \in X_h.
\end{equation*}
In order to be consistent with \eqref{eqn:semi}, we perturb the right hand side to $(f(t), \chi)$.
Then the FVE is to find $\tu_{h}\in X_h$ such that
\begin{equation}\label{fem-FVE}
 \langle\Dal\tu_{h}(t),\chi \rangle + a(\tu(t),\chi) = (f(t),  \chi) \quad \mbox{ for }t\geq 0 \quad \mbox{with }\tu_h(0)=P_h v .
\end{equation}
Thus, it corresponds to \eqref{eqn:semi} with $[\cdot,\cdot]=\langle\cdot,\cdot\rangle$. By introducing
the discrete Laplacian $-\widetilde \Delta_h:X_h\rightarrow X_h$, corresponding to the inner product
$\langle\cdot,\cdot\rangle$, defined by
\begin{equation*}
  -\langle\widetilde \Delta_h\psi,\chi\rangle = (\nabla \psi,\nabla \chi)\quad \forall\psi,\chi\in X_h,
\end{equation*}
and a projection operator $\widetilde P_h: L^2(\Omega) \rightarrow X_h$ defined by
\begin{equation*}
  \langle\widetilde P_h f, \chi\rangle = (f, \chi) \quad \forall \chi\in X_h.
\end{equation*}
In this way, the FVE method \eqref{fem-FVE} can be written with $f_h = \widetilde P_h f$ in an operator form as
\begin{equation}\label{eqn:fem-fvem}
  \Dal\tu_{h}(t)-\widetilde \Delta_h \tu_{h}(t) = f_h(t) \quad \mbox{ for }t\geq 0 \quad \mbox{with } \tu_{h}(0)=P_hv.
\end{equation}

For the analysis of the LM and FVE methods, we recall a useful quadrature error operator $Q_h: X_h\rightarrow X_h$ defined by
\begin{equation}\label{eqn:Q}
  (\nabla Q_h\chi,\nabla \psi) = \epsilon_h(\chi,\psi)
       : = [\chi,\psi]-(\chi,\psi)\quad \forall \chi,\psi\in X_h.
\end{equation}
The operator $Q_h$ represents the quadrature error in a special way. It satisfies the following error estimate
\cite[Lemma 2.4]{chatzipa-l-thomee12} for LM method and \cite[Lemma 2.2]{ChatzLazarovThomee:2013} for FVE method.
\begin{lemma}\label{lem:Q}
Let $A_h$ be $-\bDelh$ or $-\widetilde \Delta_h$, and $Q_h$ be the operator defined by \eqref{eqn:Q}. Then there holds
\begin{equation*}
  \|\nabla Q_h\chi\|_{L^2(\Omega)}+h\|A_h Q_h\chi\|_{L^2(\Omega)}\leq ch^{p+1}\|\nabla^p\chi\|_{L^2(\Omega)}
\quad \forall \chi\in X_h, ~~~p=0,1.
\end{equation*}
Furthermore, if the meshes are symmetric {\rm(}for details and illustration,
see \cite[Section 5, Fig. 2 and 3]{chatzipa-l-thomee12}{\rm)}, then there holds
\begin{equation}\label{eqn:condQ}
\|Q_h\chi\|_{L^2(\Omega)}\leq ch^2\|\chi\|_{L^2(\Omega)}\quad \forall\chi \in X_h.
\end{equation}
\end{lemma}

\def\luh{\bar u_h}

\begin{theorem}\label{lumped-mass-nonsmooth}
Let $u$ be the solution of problem \eqref{eqn:pde} and  $\luh$ be the solution of \eqref{eqn:fem-lumped}
or \eqref{eqn:fem-fvem}, respectively. Then under condition \eqref{eqn:condQ}, the following
estimates are valid for $t >0$ and  $\ell_h=|\ln h|$.
\begin{itemize}
\item[(i)] If $f\equiv0$, $v\in L^2(\Omega)$ and $v_h=P_hv$, then
\begin{equation}\label{L2-improved}
   \|\luh(t)-u(t)\|_{L^2(\Omega)}\leq ch^2 \ell_h t^{-\al} \|v\|_{L^2(\Omega)}.
\end{equation}
\item[(ii)] If $v\equiv0$, $f\in L^\infty(0,T;\dH q)$, $-1<q\leq0$, and $f_h=P_hf$, then
\begin{equation*}
 \|\luh(t) - u(t) \|_{L^2(\Omega)}  \le ch^{2+q} \ell_h^{2} \|f\|_{L^\infty(0,t;\dH q)}.
\end{equation*}
\end{itemize}
\end{theorem}
\begin{proof}
We only sketch the proof for part (i). For the analysis, we split the error $\luh(t)-u(t)$ into
\begin{equation*}
 \luh(t)-u(t) = u_h(t)- u(t) + \delta(t)
\end{equation*}
with $  \delta(t) = \luh(t)-u_h(t)$ and $u_h(t)$ being the standard Galerkin FEM solution. Upon
noting Theorem \ref{thm:error-fem} for $\|u_h -u\|_{L^2(\Omega)}$, it suffices to show
\begin{equation*}
  \|\delta(t)\|_{L^2(\Omega)} \leq ch^2 \ell_h t^{-\alpha} \|v\|_{L^2(\Omega)}.
\end{equation*}
It follows from the definitions of $u_h(t)$, $\luh(t)$, and $Q_h$ that
\begin{equation*}
  \Dal \delta(t) + A_h \delta(t) = -A_h Q_h\Dal u_h(t) \quad \mbox{ for } t > 0, \quad \mbox{with }\delta(0)=0,
\end{equation*}
where the operator $A_h$ denotes either $-\bDelh$ or $-\widetilde \Delta_h$.
By Duhamel's principle \eqref{Duhamel_o}, $\delta(t)$ can be represented by
\begin{align*}
    \delta(t) &= - \int_0^t E_h (t-s) A_h Q_h \Dal u_h(s){\d s}\\
      & = - \int_0^t A_h^{1-\epsilon}E_h (t-s) A_h^{\epsilon} Q_h \Dal u_h(s){\d s}.
\end{align*}
Then the smoothing property of $E_h$, the inverse estimate and the quadrature error assumption \eqref{eqn:condQ} imply
\begin{align*}
 \| \delta(t) \|_{L^2(\Omega)} & \le \int_0^t \|E_h (t-s)A_h^{1-\epsilon} \|   \| A_h^{\epsilon}Q_h \Dal u_h(s) \|_{L^2(\Omega)}{\d s}\\
  &\le c h^{-2\epsilon}\int_0^t  (t-s)^{\epsilon\alpha-1}    \|  Q_h \Dal u_h(s) \|_{L^2(\Omega)}{\d s} \\
 &\le  c h^{2-2\epsilon}\int_0^t  (t-s)^{\epsilon\alpha-1}    \| \Dal u_h(s) \|_{L^2(\Omega)}{\d s}.
\end{align*}
Last, the (discrete) stability result $\|\Dal u_h(t)\|_{L^2(\Omega)}\leq ct^{-\alpha}\|v_h\|_{L^2(\Omega)}$
(which follows analogously as Theorem \ref{thm:reg-u}(i)) and the $L^2(\Omega)$-stability of $P_h$
imply
\begin{equation*}
\begin{split}
 \| \delta(t) \|_{L^2(\Omega)}
 &\le  c h^{2-2\epsilon}\int_0^t  (t-s)^{\epsilon\alpha-1} s^{-\alpha}   \|   u_h(0) \|_{L^2(\Omega)}{\d s} \\
 &\le c \epsilon^{-1} h^{2-2\epsilon} t^{-\alpha} \| v_h \|_{L^2(\Omega)} \le  c \epsilon^{-1} h^{2-2\epsilon} t^{-\alpha} \| v \|_{L^2(\Omega)}.
\end{split}
\end{equation*}
Then the desired assertion follows immediately by choosing $\epsilon=1/\ell_h$.
\end{proof}

\begin{remark}
The quadrature error condition \eqref{eqn:condQ} is satisfied for symmetric meshes \cite[Section 5]{chatzipa-l-thomee12}.
If condition \eqref{eqn:condQ} does not hold, we are able to show only a suboptimal $O(h)$-convergence rate for $L^2(\Omega)$-norm
of the error \cite[Theorem 4.5]{JinLazarovZhou:SIAM2013}, which is reminiscent of that in the classical parabolic case, e.g.
\cite[Theorem 4.4]{chatzipa-l-thomee12}.
\end{remark}

Generally, the FEM analysis in the fractional case is much more delicate than the standard parabolic
case due to the less standard solution operators. Nonetheless, the results in the two cases are largely
comparable, and the overall proof strategy is often similar. The Laplace approach described above represents
only one way to analyze the spatially semidiscrete schemes. Recently, Karaa \cite{Karaa:2017} gave a
unified analysis of all three methods for the homogeneous problem based on an energy argument, which
generalizes the corresponding technique for standard parabolic problems in \cite[Chapter 3]{Thomee:2006}.
However, the analysis of the inhomogeneous case is still missing. The energy type argument is generally
more tricky in the fractional case. This is due to the nonlocality of the fractional derivative
$\partial_t^\alpha u$ and consequently that many powerful PDE tools, like integration by parts formula and product rule, are
either invalid or require substantial modification. See also \cite{PaniKaraa:2018} for some results on a related subdiffusion model.

\subsection{Illustrations and outstanding issues on semidiscrete methods}

Now we illustrate the three semidiscrete methods with very weak initial data.
\begin{example}\label{exam:fem}
Consider problem \eqref{eqn:pde} on the unit square $\Omega=(0,1)^2$ with $f=0$ and very weak initial data $v=\delta_\Gamma $,
with $\Gamma$ being the boundary of the square $[\frac14,\frac34]\times[\frac14,\frac34]$ with
$\langle \delta_\Gamma,\phi\rangle = \int_\Gamma \phi(s) \d s$. One may view $(v,\chi)$ for $\chi \in X_h \subset
\dot H^{\frac12+\epsilon}(\Omega)$ as duality pairing between the spaces
$H^{-\frac12-\epsilon}(\Omega)$ and $\dot H^{\frac12+\epsilon}(\Omega)$ for any $\epsilon >0$
so that $\delta_\Gamma \in H^{-\frac12-\epsilon}(\Omega)$.
Indeed, it follows from H\"{o}lder's inequality and trace theorem \cite{AdamsFournier:2003} that
\begin{equation*}
   \|\delta_\Gamma\|_{H^{-\frac{1}{2}-\epsilon}(\Omega)}= \sup_{\phi \in \dot H^{\frac12+\epsilon}(\Omega)}
   \frac{|\int_\Gamma \phi(s)\d s|}{\|\phi\|_{\frac12+\epsilon,\Omega}}
   \le  |\Gamma|^\frac12 \sup_{\phi \in \dot H^{\frac12+\epsilon} (\Omega)}\frac{\|\phi\|_{L^2(\Gamma)}}{\|\phi\|_{\frac12+\epsilon,\Omega}}.
 \end{equation*}
 \end{example}
The empirical convergence rate for the very weak data $\delta_\Gamma$ agrees well with the theoretically
predicted convergence rate in Remark \ref{rmk:fem-err}; see Tables \ref{tab:Galerkinweak} and
\ref{tab:Deltafun} for the standard Galerkin method and lumped mass method, respectively. In the
tables, the numbers in the bracket in the last column refer to the theoretical
rate. Interestingly, for the standard Galerkin scheme, the $L^2(\Omega)$-norm of the error exhibits
super-convergence. This is attributed to the fact that the singularity of the solution is
supported on the interface $\Gamma$ and it is aligned with the mesh. It is observed that for both the
standard Galerkin method and lumped mass FEM, the error increases as the time $t\to0^+$, which concurs with the
weak solution singularity at the initial time.

\begin{table}[h]
\caption{The errors $\|u(t)-u_h(t)\|_{L^2(\Omega)}$ of the Galerkin approximation $u_h(t)$ for
Example \ref{exam:fem} with $\al=0.5$, at $t=0.1, 0.01, 0.001$, discretized on a uniform mesh, $h = 2^{-k}$.}
\label{tab:Galerkinweak}
\begin{center}
     \begin{tabular}{|c|cccccc|}
     \hline
     $k$& $1/8$ & $1/16$ &$1/32$ &$1/64$ & $1/128$ &  rate\\
     \hline
     $t=0.001$& 5.37e-2 & 1.56e-2 & 4.40e-3 & 1.23e-3 & 3.41e-4 & $\approx 1.84$ ($1.50$)\\
     $t=0.01$  & 2.26e-2 & 6.20e-3 & 1.67e-3 & 4.46e-4 & 1.19e-4 & $\approx 1.90$ ($1.50$)\\
     $t=0.1$  & 8.33e-3 & 2.23e-3 & 5.90e-3 & 1.55e-3 & 4.10e-4 & $\approx 1.91$ ($1.50$)\\
     \hline
     \end{tabular}
\end{center}
\end{table}

\begin{table}[h]
\caption{The errors $\|u(t)-\luh(t)\|_{L^2(\Omega)}$ of the lumped mass approximation $\luh(t)$
for Example \ref{exam:fem} with $\al=0.5$, at $t=0.1, 0.01, 0.001$, discretized on a uniform mesh, $h = 2^{-k}$}
\label{tab:Deltafun}
\begin{center}
     \begin{tabular}{|c|cccccc|c|}
     \hline
     $k$& $3$ & $4$ &$5$ &$6$ & $7$ &  rate \\
   \hline
     $t=0.001$  & 1.98e-1 & 7.95e-2 & 3.00e-2 & 1.09e-2 & 3.95e-3 & $\approx 1.51$ (1.50)\\
    $t=0.01$  & 6.61e-2 & 2.56e-2& 9.51e-3 & 3.47e-3  & 1.25e-3 & $\approx 1.52$ (1.50)\\
      $t=0.1$  & 2.15e-2 & 8.13e-3 & 3.01e-3 & 1.09e-3 & 3.95e-4&$\approx 1.52$ (1.50)\\
\hline
     \end{tabular}
\end{center}
\end{table}

We end this section with some research problems. Despite the maturity of the FEM analysis,
there are still a few interesting questions on the FEMs for the model \eqref{eqn:pde} which are not well understood:
\begin{itemize}
\item[(i)] So far the analysis is mostly concerned with a time-independent coefficient, which can be treated
conveniently using the semigroup type techniques. The time dependent case requires different techniques, and the nonlocality
of the operator $\partial_t^\alpha u$ prevents a straightforward adaptation of known techniques for standard
parabolic problems  \cite{LuskinRannacher:1982}. Encouraging results in this direction using an energy
argument were established in the recent work of Mustapha \cite{Mustapha:2017},
where error estimates for the homogeneous problem were obtained.
\item[(ii)] All existing works focus on linear finite elements, and there seems no study on high-order finite
elements for nonsmooth data. It is unclear whether there are similar nonsmooth data estimates, as in the parabolic
case \cite[Chapter 3]{Thomee:2006} (see, e.g., \cite[p. 397]{QuarteroniValli:1994} for smooth data). This problem is
interesting in view of the limited smoothing property of the solution operators in space in Theorem \ref{thm:reg-u},
which has played a major role in the analysis. Thus it is of interest to develop and analyze high-order schemes in space.
\item[(iii)]The study on nonlinear subdiffusion models is rather limited, and there seems no error estimate
with respect to the data regularity, especially for nonsmooth problem data. The recent progress \cite{KovacsLiLubich:2016,
JinLiZhou:2018nm} in discrete maximal $\ell^p$ regularity results may provide useful tools for this purpose, which have proven
extremely powerful for the study of nonlinear parabolic problems. One outstanding issue seems to be sharp regularity
estimates for general problem data.
\end{itemize}

\section{Fully discrete schemes by time-stepping}\label{sec:time-stepping}
One outstanding challenge for solving the subdiffusion model lies in the accurate and efficient discretization of the
fractional derivative $\partial_t^\alpha u$. Roughly speaking, there are two predominant groups of
numerical methods for time stepping, i.e., convolution quadrature and finite difference type
methods, e.g., L1 scheme and L1-2 scheme. The former relies on approximating the (Riemann-Liouville)
fractional derivative in the Laplace domain (i.e., symbol), whereas the latter approximates the
Caputo derivative directly by piecewise polynomials. These two approaches have their pros and
cons:  convolution quadrature (CQ) is quite flexible and often much easier to analyze, since by construction, it inherits
excellent numerical stability property of the underlying schemes for ODEs, but it is often restricted
to uniform grids. The finite difference type  methods are very flexible in construction and implementation
and can easily generalize to nonuniform grids, but often challenging to analyze.
Generally, these schemes are only first-order accurate when implemented straightforwardly,
unless restrictive compatibility conditions are fulfilled. Hence, suitable corrections to the
straightforward implementation are needed in order to restore the desired high-order convergence.

In this section, we review these two popular classes of time-stepping schemes on uniform grids.
Specifically, let $\{t_n=n\tau\}_{n=0}^N$ be a uniform partition of the time interval $[0,T]$,
with a time step size $\tau=T/N$. The case of general nonuniform time grids is also of interest,
e.g., in resolving initial or interior layer, but the analysis seems not well understood at present; we refer
interested readers to the references \cite{Kopteva:2017,
LiaoLiZhang:2018,StynesORiordanGracia:2017,ZhangSunLiao:2014} for some recent progress on nonuniform grids.

\subsection{Convolution quadrature}
Convolution quadrature (CQ) was first proposed by Lubich in a series of works \cite{Lubich:1986,
Lubich:1988,Lubich:2004} for discretizing Volterra integral equations. It has been widely
applied in discretizing the Riemann-Liouville fractional derivative (see, e.g., \cite{YusteAcedo:2005,
ZengLiLiuTurner:2013,JinLazarovZhou:SISC2016}). One distinct feature is that the construction
requires only that Laplace transform of the kernel be known. Specfically, the CQ approximates
the Riemann-Liouville derivative $\partial_t^\alpha \varphi(t_n)$, which is defined by
\begin{equation*}
   ^R\partial_t^\alpha \varphi := \frac{\d}{\d t}\frac{1}{\Gamma(1-\alpha)}\int_0^t(t-s)^{-\alpha}\varphi(s)\d s
\end{equation*}
(with $\varphi(0)=0$) by a discrete convolution  (with the shorthand notation $\varphi^n=\varphi(t_n)$)
\begin{equation}\label{eqn:CQ}
\bar\partial_\tau^\alpha \varphi^n:=\frac{1}{\tau^\alpha}\sum_{j=0}^n b_j \varphi^{n-j}.
\end{equation}
The weights $\{b_j\}_{j=0}^\infty$ are the coefficients in the power series expansion
\begin{equation}\label{eqn:delta}
\delta_\tau(\zeta)^\alpha=\frac{1}{\tau^\alpha}\sum_{j=0}^\infty b_j\zeta ^j.
\end{equation}
where $\delta_\tau(\zeta)=\delta(\zeta)/\tau$ is the characteristic polynomial of a linear multistep method
for ODEs, with $\delta(\zeta)=\delta_1(\zeta)$. There are several possible choices of the characteristic
polynomial, e.g., backward differentiation formula, trapezoidal rule, Newton-Gregory method and Runge-Kutta methods. The
most popular one is the backward differentiation formula of order $k$ (BDF$k$), $k=1,\ldots,6$, for which $\delta(\zeta)$
is given by
\begin{equation*}
  \delta_\tau(\zeta ):=\frac{1}{\tau}\sum_{j=1}^k \frac {1}{j} (1-\zeta )^j,\quad j=1,2,\ldots.
\end{equation*}
The special case $k=1$, i.e., the backward Euler convolution quadrature, is commonly known as
Gr\"{u}nwald-Letnikov approximation in the literature and the coefficients $b_j$ are given explicitly by
the following recurrence relation
\begin{equation*}
  b_0 = 1, \quad b_{j}=-\frac{\alpha-j+1}{j}b_{j-1}.
\end{equation*}
Generally, the weights $b_j$ can be evaluated efficiently via recursion or discrete
Fourier transform \cite{Podlubny:1999,Sousa:2012}.

The CQ discretization first reformulates problem \eqref{eqn:pde} by the Riemann-Liouville derivative
$^R\partial_t^\alpha \varphi$ using the defining relation for the Caputo derivatives \cite[p. 91]{KilbasSrivastavaTrujillo:2006}:
$\partial_t^\alpha \varphi(t) = {^R\partial_t^\alpha}(\varphi-\varphi(0))$ into the form
\begin{equation*}
  ^R\partial_t^\alpha(u-v) - \Delta u = f.
\end{equation*}
The time stepping scheme based on the CQ for problem \eqref{eqn:pde} is to seek
approximations $U^n$, $n=1,\dots,N$, to the exact solution $u(t_n)$ by
\begin{equation}\label{eqn:BDF-CQ-0}
\bPtau^\alpha (U-v)^n  -  \Delta U^n = f(t_n) .
\end{equation}
It can be combined with space semidiscrete schemes described in Section \ref{sec:fem} to arrive at fully discrete
schemes, which are implementable. Our discussions below focus on the temporal
error for time-stepping schemes, and omit the space discretization in this part.

If the exact solution $u$ is smooth and has sufficiently many vanishing derivatives at $t=0$, then
the approximation $U^n$ converges at a rate of $O(\tau^k)$ \cite[Theorem 3.1]{Lubich:1988}. However, it
generally only exhibits a first-order accuracy when solving fractional evolution equations
 even for smooth $v$ and  $f$ \cite{CuestaLubichPalencia:2006,JinLazarovZhou:SISC2016}.
This loss of accuracy is one distinct feature for most time stepping schemes, since they are
usually derived under the assumption that the solution $u$ is sufficiently smooth, which holds
only if the problem data satisfy certain rather restrictive compatibility
conditions. In brevity, they tend to lack {robustness} with respect to the regularity of problem data.

This observation on accuracy loss has motivated some research works. For fractional ODEs, one idea
is to use starting weights \cite{Lubich:1986} to correct the CQ in discretizing $\partial_t^\alpha \varphi(t_n)$ by
\begin{equation*}
  \bar\partial_\tau^\alpha\varphi^n  = \tau^{-\alpha} \sum_{j=0}^nb_{n-j}\varphi^j + \sum_{j=0}^Mw_{n,j}\varphi^j,
\end{equation*}
where $M\in\mathbb{N}$ and the weights $w_{n,j}$ depend on $\alpha$ and $k$. The purpose of the starting term
$\sum_{j=0}^Mw_{n,j}\varphi^j $ is to capture all leading singularities so as to recover a uniform $O(\tau^k)$
rate. The weights $w_{n,j}$ have to be computed at every time step, which involves solving a linear
system with Vandermonde type matrices and may lead to instability issue (if a large $M$ is needed,
which is likely the case when $\alpha$ is close to zero). This idea works well for fractional ODEs; however,
its extension to fractional PDEs essentially seems to boil down to expanding the solution into (fractional-order)
power series in $t$, which would impose certain strong compatibility conditions on the source $f$.

The more promising idea for the model \eqref{eqn:pde} is initial correction. It corrects only
the first few steps of the schemes. This idea was first developed in \cite{LubichSloanThomee:1996}
for an integro-differential equation. Then it was applied as an abstract framework in \cite{CuestaLubichPalencia:2006}
for BDF2 in order to achieve a uniform second-order convergence for semilinear fractional
diffusion-wave equations (which is slightly different from the model \eqref{eqn:pde}) with smooth data.
Further, BDF2 CQ was extended to subdiffusion and diffusion wave equations in \cite{JinLazarovZhou:SISC2016}
and very recently also general BDF$k$ \cite{JinLiZhou:2017sisc}. In the latter work \cite{JinLiZhou:2017sisc},
by careful analysis of the error representation in the Laplace
domain, a set of simple algebraic criteria was derived.
Below we describe the correction scheme for the BDF CQ derived in \cite{JinLiZhou:2017sisc}.

To restore the $k^{\rm th}$-order accuracy for BDF$k$ CQ, we correct it at the starting $k-1$
steps by (as usual, the summation disappears if the upper index is smaller than the lower one)
\begin{equation}\label{eqn:BDF-CQ}
\left\{
\begin{aligned}
&\bPtau^\alpha (U-v)^n  -\Delta U^n = a_n^{(k)} (\Delta v+f(0))+f(t_n)+
\sum_{\ell=1}^{k-2} b_{\ell,n}^{(k)}\tau^{\ell} \partial_t^{(\ell)}f(0) ,
&&1\le n\le k-1,\\
&\bPtau^\alpha (U-v)^n  - \Delta U^n = f(t_n) , &&k\le n\le N.
\end{aligned}\right.
\end{equation}
where the coefficients $a_n^{(k)}$ and $b_{\ell,n}^{(k)}$ are given in Table \ref{tab:an}. When compared
with the vanilla scheme \eqref{eqn:BDF-CQ-0}, the additional terms are
constructed so as to improve the overall accuracy of the scheme to $O(\tau^k)$ for a general initial data
$v\in D(\Delta)$ and a possibly incompatible right-hand side $f$ \cite{JinLiZhou:2017sisc}. The only difference
between the corrected scheme \eqref{eqn:BDF-CQ} and the standard scheme \eqref{eqn:BDF-CQ-0} lies in the
correction terms at the starting $k-1$ steps for BDF$k$. Hence, the scheme \eqref{eqn:BDF-CQ} is easy to
implement. The correction is also minimal in the sense that there is no other correction scheme which uses
fewer correction steps while attaining the same accuracy. The corrected scheme \eqref{eqn:BDF-CQ} satisfies
the following error estimates \cite[Theorem 2.4]{JinLiZhou:2017sisc}.

\begin{theorem}\label{thm:conv}
Let $f\in C^{k-1}([0,T];L^2(\Omega))$ and $\int_0^t(t-s)^{\alpha-1}\|\partial_s^{(k)}f(s)\|_{L^2(\Omega)}\d s<\infty$. Then
for the solution $U^n$ to \eqref{eqn:BDF-CQ}, the following error estimates hold for any $t_n>0$.
\begin{itemize}
\item[$\rm(i)$] If $  \Delta v\in L^2(\Omega)$, then
\begin{equation*}
  \begin{aligned}
  \|U^n-u(t_n)\|_{L^2(\Omega)}  \leq & c\tau^k \bigg(t_n^{ \alpha -k }  \|f(0)+\Delta v\|_{L^2(\Omega)} + \sum_{\ell=1}^{k-1} t_n^{ \alpha+\ell -k }  \|\partial_t^{(\ell)}f(0)\|_{L^2(\Omega)}\\
  &\ \ +\int_0^{t_n}(t_n-s)^{\alpha-1}\|\partial_s^{(k)}f(s)\|_{L^2(\Omega)}\d s\bigg).
  \end{aligned}
\end{equation*}
\item[$\rm(ii)$] If $v\in L^2(\Omega)$, then
\begin{align*}
  \|U^n-u(t_n)\|_{L^2(\Omega)}  \leq c\tau^k& \bigg( t_n^{-k} \|v\|_{L^2(\Omega)} + \sum_{\ell=0}^{k-1} t_n^{ \alpha+\ell -k }  \|\partial_t^{(\ell)}f(0)\|_{L^2(\Omega)}\\
      &+\int_0^{t_n}(t_n-s)^{\alpha-1}\|\partial_s^{(k)}f(s)\|_{L^2(\Omega)}\d s\bigg).
\end{align*}
\end{itemize}
\end{theorem}

\begin{remark}
Note that the estimate
depends only on the regularity of $f$ and $v$, rather than the regularity of $u$.
Theorem \ref{thm:conv} implies that for any fixed $t_n>0$, the rate is $O(\tau^k)$ for BDF$k$ CQ.
In order to have a {uniform} $O(\tau^k)$ rate, the following compatibility conditions are needed:
\begin{equation*}
    f(0) + \Delta v = 0 \quad \mbox{and}\quad     \partial_t^{(\ell)} f(0) = 0, \quad \ell=1,\ldots,k-1.
\end{equation*}
Otherwise, the estimate
deteriorates as $t\to0$, in accordance with the regularity theory in Theorem \ref{thm:reg-u}: the solution {\rm(}and its
derivatives{\rm)} exhibits weak singularity at $t=0$.
\end{remark}

\begin{remark}
The case $k=1$ corresponds to the backward Euler CQ, and it
does not require any correction in order to achieve a first-order convergence.
\end{remark}

\begin{table}[htb!]
\caption{The coefficients $a_j^{(k)}$ and $b_{\ell,j}^{(k)}$ \cite[Tables 1 and 2]{JinLiZhou:2017sisc}.}
\label{tab:an}
\centering
     \begin{tabular}{|c|ccccc|}
\hline
       BDF$k$ &  $a_1^{(k)}$  & $a_2^{(k)}$  & $a_3^{(k)}$ & $a_4^{(k)}$ & $a_5^{(k)}$   \\[2pt]
\hline
       $k=2$       & $ \frac{1}{2}$ & & & & \\
\hline
       $k=3$          &$\frac{11}{12}$  & $-\frac{5}{12}$    &   &  &  \\[2pt]
 \hline
       $k=4$          &$\frac{31}{24}$  & $-\frac{7}{6}$  & $\frac{3}{8}$ &  &  \\[2pt]
\hline
       $k=5$          &$\frac{1181}{720}$  & $-\frac{177}{80}$   & $\frac{341}{240}$ & $-\frac{251}{720}$  & \\[2pt]
\hline
       $k=6$          &$\frac{2837}{1440}$& $-\frac{2543}{720}$   &$\frac{17}{5}$ & $-\frac{1201}{720}$ & $\frac{95}{288}$  \\[2pt]
\hline
\end{tabular}

\vspace{0.2cm}

\begin{tabular}{|c c|ccccc|}
\hline
      BDF$k$ &   & $b_{\ell,1}^{(k)}$  & $b_{\ell,2}^{(k)}$  & $b_{\ell,3}^{(k)}$ & $b_{\ell,4}^{(k)}$ & $b_{\ell,5}^{(k)}$     \\[2pt]
\hline
       $k=3$          &$\ell=1$   &$\frac{1}{12}$  &  0   &  &    &  \\[2pt]
\hline
       $k=4$          &$\ell=1$   &$\frac{1}{6}$  & $-\frac{1}{12}$    & $0$   &  &  \\[3pt]
                           &$\ell=2$   & $0$                &  $0$  & $0$   & & \\[2pt]
\hline
       $k=5$          &$\ell=1$   &$\frac{59}{240}$  & $-\frac{29}{120}$   & $\frac{19}{240}$ & $0$ &  \\[3pt]
                           &$\ell=2$   & $\frac{1}{240}$                 & $-\frac{1}{240}$  & $0$ & $0$ &  \\[3pt]
                           &$\ell=3$   &$\frac{1}{720}$ & $0$  & $0$ & $0$ & \\[2pt]
\hline
       $k=6$          &$\ell=1$   &$\frac{77}{240}$& $-\frac{7}{15}$  &$\frac{73}{240}$ & $-\frac{3}{40}$  & 0\\[3pt]
                           &$\ell=2$   & $\frac{1}{96}$             & $-\frac{1}{60}$  & $\frac{1}{160}$  & $0$ & 0 \\[3pt]
                           &$\ell=3$   & $-\frac{1}{360}$ &  $\frac{1}{720}$ & $0$ & $0$  & 0\\[3pt]
                           &$\ell=4$   & $0$ & $0$ & $0$ & $0$  & 0\\[2pt]
\hline
\end{tabular}
\end{table}

In passing, we note that not all CQ schemes require initial correction in
order to recover high-order convergence. One notable example is Runge-Kutta CQ;
see \cite{LubichOstermann:1995} for semilinear parabolic problems and
\cite{Fischer:2018} for the subdiffusion model. Further, a proper weighted average of shifted
standard Grunwald-Letnikov approximations can also lead to high-order approximations
\cite{ChenDeng:2014}. CQ schemes can exhibit superconvergence
at points that may be different from the grid points, which can also be effectively exploited
to develop high-order schemes (see \cite{Dimitrov:2014} for Grunwald-Letnikov formula). However,
the corrected versions of these approximations have not yet been developed for the general case, except for a
fractional variant of Crank-Nicolson scheme \cite{JinLiZhou:2017ima}.

\subsection{Piecewise polynomial interpolation}
Now we describe the time stepping schemes based on piecewise polynomial interpolation. These schemes are
essentially of finite difference nature, and the most prominent one is the L1 scheme. The L1 approximation
of the Caputo derivative ${\Dal} u(t_n)$ is given by \cite[Section 3]{LinXu:2007}
\begin{equation}\label{eqn:$L^1$approx}
   \begin{aligned}
     {\Dal} u(t_n) &= \frac{1}{\Gamma(1-\al)}\sum^{n-1}_{j=0}\int^{t_{j+1}}_{t_j}
        \frac{\partial u(s)}{\partial s} (t_n-s)^{-\al}\, \d s \\
     &\approx \frac{1}{\Gamma(1-\al)}\sum^{n-1}_{j=0} \frac{u(t_{j+1})-u(t_j)}{\tau}\int_{t_j}^{t_{j+1}}(t_n-s)^{-\al}\d s\\
     &=\sum_{j=0}^{n-1}b_j\frac{u(t_{n-j})-u(t_{n-j-1})}{\tau^\alpha}\\
     &=\tau^{-\al} [b_0u(t_n)-b_{n-1}u(t_0)+\sum_{j=1}^{n-1}(b_j-b_{j-1})u(t_{n-j})] =:L_1^n(u),
   \end{aligned}
 \end{equation}
where the weights $b_j$ are given by
\begin{equation*}
b_j=((j+1)^{1-\alpha}-j^{1-\alpha})/\Gamma(2-\al),\ j=0,1,\ldots,N-1.
\end{equation*}
It was shown in \cite[equation (3.3)]{LinXu:2007} and \cite[Lemma 4.1]{SunWu:2006} that the local truncation
error of the L1 approximation is bounded by
\begin{equation}\label{eqn:err-L1}
  |\partial_t^\alpha u(t_n)-L_1^n(u)|\leq c(u)\tau^{2-\alpha},
\end{equation}
where the constant $c(u)$ depends on $\|u\|_{C^2([0,T])}$. Thus, it requires that the solution $u$ be twice
continuously differentiable in time. Since its first appearance, the L1 scheme
has been widely used in practice, and currently it is one of the most popular and successful numerical
methods for solving the model \eqref{eqn:pde}. With the L1 scheme in time,
we arrive at the following time stepping scheme: Given $U^0=v$, find $U^n\in \dH1$ for $n=1,2,\ldots,N$
\begin{equation}\label{eqn:fullyl1}
    L_1^n(U) -\Delta U^n=  f(t_n).
\end{equation}

We have the following temporal error estimate for the scheme \eqref{eqn:fullyl1} \cite{JinLazarovZhou:2016ima,
JinLiZhou:nonlinear}. This is achieved by means of discrete Laplace transform, and it is rather technical,
since the discrete Laplace transform of the weights $b_j$ involves the wieldy polylogarithmic function.
See also \cite{JinZhou:2017} for a different analysis via an energy argument. Formally, the error estimate
is nearly identical with that for the backward Euler CQ.
Thus, in stark contrast to the $O(\tau^{2-\alpha})$ rate expected from the local truncation error \eqref{eqn:err-L1}
for smooth solutions, the L1 scheme is generally only first-order accurate, even for smooth initial data or source term.
\begin{theorem}\label{thm:error_fullyl1-nonsmooth}
Let $u$ and $U^n$ be the solutions of problems \eqref{eqn:pde} and \eqref{eqn:fullyl1},
respectively. Then there holds
\begin{equation*}
  \|u(t_n)-U^n\|_{L^2(\Omega)} \leq c\tau t_n^{\beta\alpha-1}\|(-\Delta)^\beta u_0\|_{L^2(\Omega)}+ c\tau\bigg(t_n^{\alpha-1}\|f(0)\|_{L^2(\Omega)}+\int_0^{t_n}(t_n-s)^{\alpha-1}\|f'(s)\|_{L^2(\Omega)}\d s\bigg).
\end{equation*}
\end{theorem}

Very recently, a corrected L1 scheme was developed  by Yan et al \cite{YanKhanFord:2018} (see
also \cite{XingYan:2018,FordYan:2017} for related works from the group). The corrected scheme is given by
\begin{equation}\label{eqn:fullyl1-corrected}
  \left\{\begin{aligned}
   L_1^1(U) -\Delta U^1 - \tfrac{1}{2}
 \Delta U^0 & =  f(t_1)+ \tfrac{1}{2}f(0), \quad n=1\\
   L_1^n(U) -\Delta U^n & =  f(t_n),\quad n\geq 2.
  \end{aligned}\right.
\end{equation}
It is noteworthy that it requires only correcting the first step, and incidentally,
the correction term is identical with that for BDF2 CQ.
Then the following error estimate holds for the corrected scheme.
Note that the stated regularity requirement on the source term $f$ may not
be optimal for $\alpha>1/2$.
\begin{theorem}\label{thm:error_fullyl1-corrected}
Let $u$ and $U^n$ be the solutions of problems \eqref{eqn:pde} and \eqref{eqn:fullyl1-corrected}, respectively. Then there holds
\begin{align*}
  \|u(t_n)-U^n\|_{L^2(\Omega)} &\le c \tau^{2-\alpha}\bigg(t_n^{(\beta+1)\alpha-2} \|(-\Delta)^\beta v\|_{L^2(\Omega)} + t_n^{2\alpha-2}\|f(0)\|_{L^2(\Omega)} + t_n^{2\alpha-1}\|f'(0)\|_{L^2(\Omega)}\\
&\quad +\int_0^{t_n}(t_n-s)^{2\alpha-1}\|f''(s)\|_{L^2(\Omega)}\d s\bigg).
\end{align*}
\end{theorem}

There have been several important efforts in extending the L1 scheme to high-order schemes by
using high-order local polynomials \cite{LvXu:2016,GaoSunZhang:2014,MustaphaAbdallahFurati:2014} and superconvergent points
\cite{Alikhanov:2015}. For example, the L1-2 scheme due to Gao et al \cite{GaoSunZhang:2014} applies a piecewise linear approximation on the
first subinterval, and a quadratic approximation on the other subintervals to
improve the numerical accuracy. However, the performance of these methods for nonsmooth data is not fully understood.

Besides, Mustapha and McLean developed several discontinuous Galerkin methods \cite{McLeanMustapha:2009,MustaphaMcLean:2011,Mustapha:2015}
for a variant of the model \eqref{eqn:pde}:
\begin{equation*}
  \partial_t u - {^R\partial_t^{1-\alpha}}\Delta u = f,
\end{equation*}
with suitable boundary and initial conditions. Formally, this model can be derived by applying the Riemann-Liouville
operator $^R\partial_t^{1-\alpha}$ to both sides of the equation in \eqref{eqn:pde}. The resulting schemes are similar to piecewise
polynomial interpolation described above. However, the nonsmooth error estimates are mostly unavailable,
except for the piecewise constant discontinuous Galerkin method (for the homogeneous problem)
\cite{McLeanMustapha:2015,YangYanFord:2018}; see also \cite{GunzburgerWang:2018} for a Crank-Nicolson
type scheme for a related model.

\subsection{Illustrations and outstanding issues}

Now we illustrate the performance of the corrected time stepping schemes.
\begin{example}\label{exam:time-stepping}
Consider problem \eqref{eqn:pde} on $\Omega=(0,1)$ with $v=x\sin(2 \pi x) \in \dH 2$
and $f= 0$.
\end{example}

In Table  \ref{tab:v-nocorrect} we present the error $\|u_h(t_N) - u_h^N\|_{L^2(\Omega)}$ at $t_N=1$. The numerical results
show only a first-order empirical convergence rate, for all standard
BDF$k$ CQ, $k\geq2$, which shows clearly the lack of robustness of the naive CQ scheme \eqref{eqn:BDF-CQ-0} with
respect to problem data regularity, despite the good regularity of the initial data $v$. In sharp contrast,
the corrected scheme \eqref{eqn:BDF-CQ} can achieve the desired convergence rate; see Table \ref{tab:v-correct-smooth}.
These observations remain valid for the L1 scheme and its corrected version; see Tables
\ref{tab:uncorrect-smooth-L1} and \ref{tab:correct-smooth-L1}. It is worth noting that the desired rate for the corrected L1 scheme only
kicks in at a relatively small time step size, and its precise mechanism remains unclear. These results
show clearly the effectiveness of the idea of initial correction for restoring the desired high-order convergence.

\begin{table}[htb!]
\caption{The $L^2$-norm error for Example \ref{exam:time-stepping} at $t_N=1$, by
the scheme \eqref{eqn:BDF-CQ-0} with $h=1/100$.}\label{tab:v-nocorrect}
\vskip-10pt
\centering
     \begin{tabular}{|c|c|ccccc|c|}
     \hline
      $\alpha$ &  $N$  &$50$ &$100$ &$200$ & $400$ & $800$  &rate \\
      \hline
             & BDF2  &4.94e-3 &2.48e-3 &1.24e-3 &6.20e-4 &3.10e-4 &$\approx$ 1.00 ($1.00$)\\
             & BDF3  &4.99e-3 &2.49e-3 &1.24e-3 &6.21e-4 &3.11e-4 &$\approx$ 1.00 ($1.00$)\\
     $ 0.5$  & BDF4  &4.99e-3 &2.49e-3 &1.24e-3 &6.21e-4 &3.11e-4 &$\approx$ 1.00 ($1.00$)\\
             & BDF5  &4.99e-3 &2.49e-3 &1.24e-3 &6.21e-4 &3.11e-4 &$\approx$ 1.00 ($1.00$)\\
             & BDF6  &4.96e-3 &2.49e-3 &1.24e-3 &6.21e-4 &3.11e-4 &$\approx$ 1.00 ($1.00$)\\
      \hline
     \end{tabular}
\end{table}

\begin{table}[htb!]
\caption{The $L^2$-norm error for Example \ref{exam:time-stepping} at $t_N=1$, by the corrected scheme
\eqref{eqn:BDF-CQ} with $h=1/100$.}
\label{tab:v-correct-smooth}
\centering
     \begin{tabular}{|c|c|ccccc|c|}
     \hline
      $\alpha$ &  $k\backslash N$  &$50$   &$100$ &$200$ &$400$ & $800$   &rate \\
     \hline
             & 2  &5.87e-5 &1.45e-5 &3.59e-6 &8.95e-7 &2.23e-7 &$\approx$ 2.00 (2.00)\\
             & 3  &2.39e-6 &2.88e-7 &3.53e-8 &4.38e-9 &5.45e-10 &$\approx$ 3.00 (3.00)\\
    $ 0.25$  & 4  &1.49e-7 &8.72e-9 &5.27e-10 &3.24e-11 &2.01e-12  &$\approx$ 4.02 (4.00)\\
             & 5  &1.33e-8 &3.57e-10 &1.06e-11 &3.22e-13 &9.91e-15 &$\approx$ 5.02 (5.00)\\
             & 6  &1.12e-5 &1.54e-9 &2.68e-13 &4.02e-15 &6.16e-17 &$\approx$ 6.04 (6.00)\\
      \hline
             & 2  &1.77e-4 &4.34e-5 &1.08e-5 &2.68e-6 &6.69e-7 &$\approx$ 2.00 (2.00)\\
             & 3  &7.85e-6 &9.44e-7 &1.16e-7 &1.43e-8 &1.78e-9 &$\approx$ 3.01 (3.00)\\
    $ 0.5$   & 4  &5.23e-7 &3.04e-8 &1.83e-9 &1.12e-10 &6.97e-12 &$\approx$ 4.02 (4.00)\\
             & 5  &4.86e-8 &1.30e-9 &3.85e-11 &1.17e-12 &3.60e-14 &$\approx$ 5.03 (5.00)\\
             & 6  &2.82e-5 &2.99e-9 &1.01e-12 &1.51e-14 &2.32e-16 &$\approx$ 6.05 (6.00)\\
      \hline
             & 2  &4.58e-4 &1.12e-4 &2.78e-5 &6.92e-6 &1.73e-6 &$\approx$ 2.00 (2.00)\\
             & 3  &2.39e-5 &2.85e-6 &3.49e-7 &4.31e-8 &5.36e-9 &$\approx$ 3.01 (3.00)\\
    $ 0.75$  & 4  &1.80e-6 &1.04e-7 &6.22e-9 &3.81e-10 &2.36e-11 &$\approx$ 4.02 (4.00)\\
             & 5  &2.51e-7 &4.90e-9 &1.44e-10 &4.35e-12 &1.34e-13 &$\approx$ 5.03 (5.00)\\
             & 6  &1.65e-3 &4.20e-7 &4.17e-12 &6.10e-14 &9.31e-16 &$\approx$ 6.06 (6.00)\\
      \hline
     \end{tabular}
\end{table}

\begin{table}[htb!]
\caption{The $L^2$-norm error for Example \ref{exam:time-stepping} at $t_N=1$, by the L1 scheme
\eqref{eqn:BDF-CQ} with $h=1/100$.}
\label{tab:uncorrect-smooth-L1}
\centering
     \begin{tabular}{|c|ccccc|c|}
     \hline
      $\alpha \backslash N$  &$50$   &$100$ &$200$ &$400$ & $800$   &rate \\
     \hline
      $0.3$         &2.40e-3 &1.19e-3 &5.96e-4 &2.98e-4 &1.49e-4 &$\approx$ 1.01 (1.00)\\
      \hline
       $0.5$         &5.09e-3 &2.52e-3 &1.25e-3 &6.25e-4 &3.12e-4 &$\approx$ 1.02 (1.00)\\
      \hline
       $0.7$        &9.04e-3  &4.42e-3 &2.18e-3 &1.08e-3 &5.33e-4 &$\approx$ 1.01 (1.00)\\
      \hline
     \end{tabular}
\end{table}

\begin{table}[htb!]
\caption{The $L^2$-norm error for Example \ref{exam:time-stepping} at $t_N=0.01$, corrected L1 scheme,
 $h=1/100$ and $N=1000\times 2^{k}$.}
\label{tab:correct-smooth-L1}
\centering
     \begin{tabular}{|c|cccccc|c|}
     \hline
      $\alpha \backslash k$  &$1$   &$2$ &$3$ &$4$ & $5$ & $6$  &rate \\
     \hline
      $0.3$         &7.94e-8 &2.79e-8 &9.67e-9 &3.28e-9 &1.09e-9 &3.56e-10 &$\approx$ 1.63 (1.70)\\
      \hline
       $0.5$        &1.90e-6 &6.93e-7 &2.50e-7 &8.95e-8 &3.19e-8 &1.14e-8  &$\approx$ 1.49 (1.50)\\
      \hline
       $0.7$        &1.97e-5 &8.06e-6 &3.29e-6 &1.34e-6 &5.44e-7 &2.21e-7  &$\approx$ 1.30 (1.30)\\
      \hline
     \end{tabular}
\end{table}

We conclude this section with two research directions on time stepping schemes that need/deserve further investigation.
\begin{itemize}
\item[(i)] Nonsmooth error analysis for time-stepping schemes is still in its infancy. So far all
known results are only for uniform grids, and all the proofs rely essentially on Laplace transform.
It is of immense interest to develop energy type arguments that yield nonsmooth data error estimates,
which might allow deriving results for nonuniform grids. Likewise, correction schemes are also only developed for uniform grids. This is partially due
to the fact that the current construction of corrections essentially relies on Laplace transform of the kernel and its discrete analogue.
\item[(ii)] The error estimates are only derived for problems with a time-independent elliptic operator, and there
are no analogous results for time-dependent elliptic operators, including time-dependent coefficient
and certain nonlinear problems.
\end{itemize}

\section{Space-time formulations}\label{sec:space-time}
Due to the nonlocality of the fractional derivative $\partial_t^\alpha u$, at each time step
one has to use the numerical solutions at all preceding time levels. Thus, the advantages of
time stepping schemes, when compared to space-time schemes, are not as pronounced as in the case of
standard parabolic problems, and it is natural to consider space-time discretization.
Naturally, any such construction would rely on a proper variational formulation of the fractional derivative,
which is only well understood for the Riemann-Liouville derivative $^R\partial_t^\alpha u$ at
present. Thus, the idea so far is mostly restricted to problem \eqref{eqn:pde} with $v=0$, for
which the Riemann-Liouville and Caputo derivatives coincide, and we shall not distinguish the
two fractional derivatives in this section.
Throughout, let $I=(0,T)$, and the space $\widetilde{H}_L^s(I) $ consists of functions whose
extension by zero belong to $H^{s}(-\infty,T)$. On the cylindrical domain $Q_T=\Omega\times I$,
we denote the  $L^2(Q_T)$-inner product by $(\cdot,\cdot)_{L^2(Q_T)}$.

\subsection{Standard Galerkin formulation}\label{ssec:space-time-Galerkin}
In an influential work, Li and Xu \cite{LiXu:2009} proposed a first rigorous space-time formulation
for problem \eqref{eqn:pde}, which was extended and refined by many other researchers (see, e.g.,
\cite{ZayernouriAinsworth:2015,HouHasanXu:2018} and the references therein). For any $s\in[0,1]$,
we denote by
$$B^s(Q_T)=H^s(I;L^2(\Omega))\cap L^2(I;H_0^1(\Omega)),$$
with a norm defined by
\begin{equation*}
\|v\|_{B^s(Q_T)}^2=\|v\|_{H^s(I;L^2(\Omega))}^2+\|v|_{L^2(I;H^1(\Omega))}^2.
\end{equation*}
The foundation of the method is the following important identity \cite[Lemma 2.6]{LiXu:2009}
\begin{equation}\label{eqn:int-part}
  ({\DDR0{\alpha}}w,v)_{L^2(I)} = ({\DDR0{\frac{\alpha}{2}}}w,{\DDR1{\frac{\alpha}{2}}}v)_{L^2(I)}\quad \forall w\in \widetilde H_L^1(I),v\in \widetilde H_L^\frac{\alpha}{2}(I),
\end{equation}
where $\DDR0\gamma w$ and $\DDR1\gamma w$ denote the left-sided and right-sided Riemann-Liouville
fractional derivatives, respectively, and for $\gamma\in(0,1)$, and are defined by
\begin{align*}
  {\DDR0\gamma w}(t) &= \frac{\d}{\d t}\frac{1}{\Gamma(1-\gamma)}\int_0^t(t-s)^{-\gamma}w(s)\d s, \\
  {\DDR1\gamma w}(t) &= -\frac{\d}{\d t}\frac{1}{\Gamma(1-\gamma)}\int_t^T(s-t)^{-\gamma}w(s)\d s.
\end{align*}
By multiplying both sides of problem \eqref{eqn:pde} with $v\in B^\frac{\alpha}{2}(Q_T)$, integrating
over the cylindrical domain $Q_T$, applying the formula \eqref{eqn:int-part} in time and integration by parts
in space, we obtain the following bilinear form on the space $B^\frac{\alpha}{2}(Q_T)$:
\begin{equation*}
  a(u,v) = ({\DDR0{\frac{\alpha}{2}}}u,{\DDR1{\frac{\alpha}{2}}}v)_{L^2(Q_T)} + (\nabla u,\nabla v)_{L^2(Q_T)}.
\end{equation*}
Hence, the weak formulation of problem \eqref{eqn:pde} is given by: for $f\in L^2(Q_T)$,
find $u\in B^\frac{\alpha}{2}(Q_T)$ such that
\begin{equation}\label{eqn:weak-LiXu}
  a(u,v) = (f,v)_{L^2(Q_T)}\quad \forall v\in B^\frac{\alpha}{2}(Q_T).
\end{equation}
Clearly, the bilinear form $a(\cdot,\cdot)$ is not symmetric, since the Riemann-Liouville derivatives
$\DDR0\gamma u(t)$ and $\DDR1\gamma u(t)$ differ. Nonetheless, it is continuous on the space $B^\frac\alpha2(Q_T)$.
Further, since the inner product $({\DDR0{\frac{\alpha}{2}}}v,\ {\DDR1{\frac{\alpha}{2}}}v)_{L^2(I)}$
involving Riemann-Liouville derivatives actually induces an equivalent norm on the space
$H^\frac{\alpha}{2}(I)$ (e.g., by means of Fourier transform) (see, e.g., \cite[Lemma 2.5]{LiXu:2009}
and \cite[Lemma 4.2]{JinLazarovPasciakRundell:2015}):
\begin{equation*}
  ({\DDR0{\frac{\alpha}{2}}}v,\ {\DDR1{\frac{\alpha}{2}}}v)_{L^2(I)}\geq \|v\|_{H^\frac{\alpha}{2}(I)}^2,
\end{equation*}
we have the following coercivity of the bilinear form $a(\cdot,\cdot)$
\begin{equation*}
  a(u,u) \geq c\|u\|_{B^\frac\alpha2(Q_T)}^2.
\end{equation*}
Then the well-posedness of the weak formulation \eqref{eqn:weak-LiXu} follows directly from Lax-Milgram theorem.

To discretize the weak formulation, Li and Xu \cite{LiXu:2009} employed a spectral approximation for the case
of one-dimensional spatial domain $\Omega$. Specifically, let $P_N(I)$
(respectively $P_M(\Omega)$) be the polynomial space of degree less than or equal to $N$ (respectively $M$)
with respect to $t$ (respectively $x$). For the spectral approximation in space, the authors employ the space
$P_M^0(\Omega):=P_M(\Omega)\cap H_0^1(\Omega)$, and since $v=0$, it is natural to construct the approximation
space (in time):
\begin{equation*}
  P_N^E(I) :=\{v\in P_N(I): v(0)=0\}.
\end{equation*}
Then for a given pair of integers $M,N$, let $L:=(M,N)$ and $S_L:=P_M^0(\Omega)\otimes P_N^E(I)\subset B^\frac{\alpha}{2}(Q_T)$. The
space-time spectral Galerkin approximation to problem \eqref{eqn:pde} reads: find $u_L\in S_L$ such that
\begin{equation*}
   a(u_L,v_L) = (f,v_L)_{L^2(Q_T)}\quad \forall v_L\in S_L.
\end{equation*}
The well-posedness of the discrete problem follows directly from Lax-Milgram theorem. The authors also provided
optimal error estimates in the energy norm. However, the $L^2(Q_T)$ error estimate for the approximation
remains unclear, since the regularity of the adjoint problem is not well understood. Clearly the construction
extends directly to rectangular domains.

Note that in order to achieve high-order convergence, the standard polynomial approximation space requires high regularity of the solution $u$ in time,
which is nontrivial to ensure a priori, in view of the limited smoothing property of the solution
operators. Hence, recently, there have been immense interest in developing schemes that can take care of the
solution singularity directly. In the context of space-time formulations, singularity enriched
trial and/or test spaces, e.g., generalized Jacobi polynomials \cite{ChenShenWang:2016} (including
 Jacobi poly-fractonomials \cite{ZayernouriKarniadakis:2013}) and M\"untz polynomials \cite{HouHasanXu:2018},
are extremely promising and have demonstrated very encouraging numerical results. However, the
rigorous convergence analysis of such schemes can be very challenging, and is mostly missing
for nonsmooth problem data.

\subsection{Petrov-Galerkin formulation}Now we introduce a Petrov-Galerkin
formulation recently developed in \cite{DuanJinLazarovPasciakZhou:2018}.
Let $V(Q_T)=L^2(I;H_0^1(\Omega))$ and by $V^*(Q_T)$ its dual, and
for any $0<s<1$, define the space $B^s(Q_T)$ by
\begin{equation*}
   B^s(Q_T)=\widetilde{H}_L^s(I; H^{-1}(\Omega)) \cap L^2(I;H_0^{1}(\Omega)).
\end{equation*}
The space is endowed with the norm
\begin{equation*}
 \|v\|^2_{B^s(Q_T)} = \| {\partial_t^s} v \|^2_{V^*(Q_T)} + (\nabla v,\nabla v)_{L^2(Q_T)}.
\end{equation*}
Here we have slightly abused the notation $B^s(Q_T)$ since it differs from that in Section \ref{ssec:space-time-Galerkin}.
Then we define the bilinear form $a(\cdot,\cdot):B^\alpha(Q_T) \times V(Q_T) \to \mathbb{R}$ by
\begin{equation*}
a(v, \phi) := ({\partial_t^\alpha} v, \phi)_{L^2(Q_T)} + ( \nabla u, \nabla v )_{L^2(Q_T)}.
\end{equation*}
The Petrov-Galerkin weak formulation of problem \eqref{eqn:pde} reads: find $ u \in {B}^\alpha(Q_T)$  such that
\begin{equation}\label{eqn:BV-weak}
 	a(u, \phi) = ( f, \phi )_{L^2(Q_T)} \quad \forall \phi \in V(Q_T).
\end{equation}
The bilinear form $a(\cdot,\cdot)$ is continuous on $B^\alpha(Q_T) \times V(Q_T)$, and
 it satisfies the following inf-sup condition
\begin{equation*}
	\sup_{\phi \in V(Q_T)} \frac{a(v,\phi) }{\|\phi \|_{V(Q_T)}} \ge \|v \|_{B^\alpha(Q_T)}\quad \forall v\in B^\alpha(Q_T)
\end{equation*}
and a compatibility condition, i.e., $\sup_{v \in B^\alpha (Q_T)} a(v,\phi) >0$ for any $0\neq \phi\in V(Q_T)$
\cite[Lemma 2.4]{DuanJinLazarovPasciakZhou:2018}.
Thus the well-posedness of the space-time formulation follows directly from the Babuska-Brezzi theory.

Now the development of a novel Petrov-Galerkin method is based on the following idea.
Let ${X}_h$ be the space of continuous piecewise linear
functions on a quasi-uniform shape regular triangulation $\mathcal{T}_h$ of the domain $\Omega$. Also, take a
uniform partition of the time interval $I$ with grid points $t_n=n \tau$, $n=0,\ldots,N$, and time
step-size $\tau=T/N$. Following \cite{JinLazarovZhou:2016sinum},
define a set of ``fractionalized'' piecewise constant basis functions $ \phi_n(t)$, $n =1,\ldots,N$, by
\begin{equation*}
\phi_n(t) = (t - t_{n-1})^{\alpha}\chi_{[t_{n-1}, T]}(t),
\end{equation*}
where $\chi_S$ denotes the characteristic function of the set $S$. It is
easy to verify that
$$
\phi_n(t)=\Gamma(\alpha +1){_0I_t^{\alpha}}\chi_{[t_{n-1}, T]}(t) \quad \mbox{ and } \quad
\partial_t^\alpha \phi_k (t) =\Gamma(\alpha +1) \chi_{[t_{n-1}, T]}(t).
$$
Clearly, $\phi_k \in \widetilde H_L^{\alpha+s}(0,T)$ for any $s\in[0,1/2)$.

Further, we introduce the following two spaces
\begin{equation*}
V_\tau= \text{span}(\{\phi_n(t) \}_{n=1}^N)
\quad\mbox{and}\quad W_\tau := \text{span}(\{ \chi_{[t_{n-1}, T]}(t)  \}_{n=1}^N).
\end{equation*}
Then the solution space ${B}_{h,\tau}^\alpha \subset B^\alpha(Q_T)$ and the test space ${V}_{h,\tau}(Q_T)
\subset  V(Q_T)$ are respectively defined by ${B}_{h,\tau}^\alpha (Q_T):= X_h \otimes {V}_\tau$
and ${V}_{h,\tau}(Q_T) := {X}_h \otimes {W}_\tau$. The space-time Petrov-Galerkin FEM problem of
\eqref{eqn:pde} reads: given $f\in V^*(Q_T)$, find $u_{h\tau} \in {B}_{h,\tau}^\alpha(Q_T)$ such that
\begin{equation}\label{eqn:BV-weak-h}
 a(u_{h \tau}, \phi)   = ( f, \phi )_{L^2(Q_T)} \quad \forall \phi \in {V}_{h,\tau}(Q_T).
\end{equation}
Algorithmically, it leads to a time-stepping like scheme, and thus admits an efficient practical implementation.
The existence and the stability of the solution  $u_{h \tau}$ follows from the discrete inf-sup condition
\cite[Lemma 3.3]{DuanJinLazarovPasciakZhou:2018}
\begin{equation*}
 \sup_{\phi \in{V}_{h,\tau}(Q_T)} \frac{a(v,\phi) }{\|\phi \|_{V(Q_T)}} \ge c_\alpha\|v\|_{B^\alpha(Q_T)}
               \quad \forall   v \in {B}_{h,\tau}^\alpha(Q_T).
\end{equation*}
This condition was shown using the $L^2(I)$ stability of the projection operator from ${V}_\tau$
to ${W}_\tau$. It is interesting to note that the constant in the $L^2(I)$-stability of  the operator
depends on the fractional order $\alpha$ and deteriorates as $\alpha \to 1$. Note that for standard parabolic
problems ($\alpha=1$), it depends on the time step size $\tau$, leading to an undesirable CFL-condition, a
fact shown in \cite{LarssonMolteni:2017}. This indicates one significant difference between the fractional model
and the standard parabolic model in the context of space-time formulations. In passing, we also note a different Petrov-Galerkin
formulation proposed very recently in \cite{Karkulik:2018}, whose numerical realization, however, has not
been carried out yet and needs computational verification.

Next, we give two error estimates for the space-time Petrov-Galerkin approximation $u_{h\tau}$,
\cite[Theorems 5.2 and 5.3]{DuanJinLazarovPasciakZhou:2018}, in $ H_L^{s}(0,T;L^2(\Omega ))$- and $L^2(Q_T)$-norms, respectively.
\begin{theorem}\label{thm:err-energy}
Let $f \in \widetilde H_L^{s}(0,T;L^2(\Omega ))$ with $0 \le s \le 1$, and $u$ and
$u_{h \tau}$ be the solutions of \eqref{eqn:BV-weak} and
\eqref{eqn:BV-weak-h}, respectively. Then there holds
\begin{align*}
	\| u-u_{h \tau}\|_{B^\alpha(Q_T)}&\le c ( \tau^s+ h)\|f\|_{\widetilde{H}_L^s( {0,T;L^2 (\Omega )})},\\
  	\| u-u_{h \tau} \|_{{L^2}({Q_T})}&\le c ({\tau ^{\alpha+s}} +  h^2 )\| f\|_{\widetilde{H} _L^s(0,T;L^2(\Omega))}.
\end{align*}
\end{theorem}

\subsection{Numerical illustrations, comments and research questions}
Now we present some numerical results to show the performance of the space-time Petrov-Galerkin FEM.
\begin{example}\label{exam:space-time-1}
Consider problem \eqref{eqn:pde} on the unit square domain $\Omega=(0,1)^2$ with $v\equiv 0$ and
\begin{itemize}
\item[(a)] $f= (e^t-1)x(1-x)y(1-y)\in \widetilde{H}_L^1( {0,T;L^2 (\Omega )})$;
\item[(b)] $f= t^{-0.2}x(1-x)y(1-y)\in\widetilde{H}_L^s( {0,T;L^2 (\Omega )})$, for any $s<0.3$.
\end{itemize}
The error $\|u_h - u_{h \tau}\|_{{L^2}(Q_T)}$ with $T=1$ is presented in Table \ref{tab:space-time-1}.
Due to the compatibility of the data, we observe an error in  ${L^2}(Q_T)$-norm of order
$O(\tau^{\alpha+1})$ and $O(\tau^{\alpha+0.3})$ for cases {\rm(}a{\rm)} and {\rm(}b{\rm)}, respectively,
which fully supports the theoretical results in Theorem \ref{thm:err-energy}.
\end{example}

\begin{table}[htb!]
\caption{The ${L^2}(Q_T)$-norm error for Example \ref{exam:space-time-1}, space-time PGFEM scheme with $\tau=T/N$ and $h=1/200$.}\label{tab:space-time-1}
\vskip-10pt
\centering
     \begin{tabular}{|c|c|ccccc|c|}
     \hline
      &$\alpha\backslash N$   &$20$ &$40$ & $80$ & $160$ & $320$ &rate \\
      \hline
              & 0.3  &1.02e-5 &4.16e-6 &1.70e-6 &7.03e-7 &2.96e-7 &$\approx$ 1.26 ($1.30$)\\
       $(a)$  & 0.5  &4.30e-6 &1.53e-6 &5.51e-7 &2.02e-7 &7.73e-8 &$\approx$ 1.45 ($1.50$)\\
              & 0.7  &2.05e-6 &6.43e-7 &1.98e-7 &6.17e-8 &2.00e-8 &$\approx$ 1.63 ($1.70$)\\
              \hline
              & 0.3  &3.20e-5 &2.50e-5 &1.93e-4 &1.48e-4 &1.13e-4 &$\approx$ 0.40 ($0.60$)\\
       $(b)$  & 0.5  &2.93e-4 &1.99e-4 &1.31e-4 &8.37e-5 &5.24e-5 &$\approx$ 0.67 ($0.80$)\\
              & 0.7  &2.15e-4 &1.14e-4 &5.64e-5 &2.72e-5 &1.31e-5 &$\approx$ 1.05 ($1.00$)\\
      \hline
     \end{tabular}
\end{table}

We conclude this sections with two important research problems on space-time formulations.
\begin{itemize}
  \item[(i)] The development of space-time formulations relies crucially on proper variational formulations
  for the fractional derivative, and this is relatively well understood for the Riemann-Liouville fractional derivative
  but not yet for the Caputo one. This is largely the main reason for the restriction to the case of a
  zero initial data. It is of much interest to develop techniques for handling nonzero initial data in the Caputo case,
  especially nonsmooth initial data.
  \item[(ii)] Nonpolynomial type approximation spaces for trial and test lead to interesting new schemes, supported by extremely promising
numerical results. However, the performance may depend strongly on the exponent of the fractional powers, and it would be of
much interest to develop strategies to adapt the algorithmic parameter automatically. Many theoretical questions surrounding such
schemes, of either Galerkin or Petrov-Galerkin type, are largely open.
\end{itemize}

\section{Concluding remarks}\label{sec:conclus}
In this paper, we have concisely surveyed relevant results on the topic of numerical methods of the
subdiffusion problem with nonsmooth problem data, with a focus on the state of the art of the following aspects:
regularity theory, finite element discretization, time-stepping schemes and space-time formulations.
We compared the theoretical results with that for standard parabolic problems, and provided illustrative
numerical results.  We also outlined a few interesting research problems that would lead to further developments and theoretical
understanding, and pointed out the most relevant references. Thus, it may serve as a brief introduction to this fast
growing area of numerical analysis.

The subdiffusion model represents one of the simplest models in the zoology of fractional diffusion or
anomalous diffusion. The authors believe that many of the analysis may be extended to more complex ones,
e.g., diffusion wave model, multi-term, distributed-order model, tempered subdiffusion,
nonsingular Caputo-Fabrizio fractional derivative, and space-time fractional models. However, these complex models have
scarcely been studied in the context of nonsmooth problem data, and their distinct features remain
largely to be explored both analytically and numerically.


\bibliographystyle{abbrv}
\bibliography{frac}
\end{document}